\documentclass[12pt]{amsart}
\usepackage{graphicx}
\newtheorem{thm}{Theorem}[section]
\newtheorem{cor}[thm]{Corollary}
\newtheorem{lem}[thm]{Lemma}
\newtheorem{prop}[thm]{Proposition}
\theoremstyle{definition}

\theoremstyle{remark}

\numberwithin{equation}{section}

\begin{document}

\title[Balanced Metrics and Chow Stability]{Balanced Metrics and Chow Stability of Projective Bundles over K\"ahler Manifolds II}%
\author{REZA SEYYEDALI}%
\address{UCI, Department of Mathematics}%
\email{rseyyeda@math.uci.edu}%

\thanks{}%
\subjclass{}%
\keywords{}%

\date{April 20, 2011}

\begin{abstract}

In the previous article (\cite{S}), we proved that  slope stability of  a holomorphic vector bundle $E$ over a polarized manifold $(X,L)$ implies  Chow stability  of $(\mathbb{P}E^*,\mathcal{O}_{\mathbb{P}E^*}(1)\otimes \pi^* L^k)$  for $k \gg 0$ if the base manifold has no nontrivial holomorphic vector field and admits a constant scalar curvature metric in the class of  $2\pi  c_{1}(L)$. In this article using asymptotic expansions of Bergman kernel on  $\textrm{Sym}^d E$, we generalize the main theorem of \cite{S} to polarizations   $(\mathbb{P}E^*,\mathcal{O}_{\mathbb{P}E^*}(d)\otimes \pi^* L^k)$ for $k \gg 0$, where $d$ is a positive integer.

\end{abstract}
\maketitle
\pagenumbering{arabic}

\newtheorem*{Mthm}{Main Theorem}
\newtheorem{Thm}{Theorem}[section]
\newtheorem{Prop}[Thm]{Proposition}
\newtheorem{Lem}[Thm]{Lemma}
\newtheorem{Cor}[Thm]{Corollary}
\newtheorem{Def}[Thm]{Definition}
\newtheorem{Guess}[Thm]{Conjecture}
\newtheorem{Ex}[Thm]{Example}
\newtheorem{Rmk}{Remark}
\newtheorem{Not}{Notation}
\def\thesection{\arabic{section}}
\renewcommand{\theThm} {\thesection.\arabic{Thm}}

\section{Introduction}

In \cite{S}, we prove that if  a holomorphic vector bundle  $E$ over a polarized
algebraic manifold  $(X,L)$  is  Mumford stable and if $(X,L)$ admits a constant
scalar curvature metric and has discrete automorphism group, then $(\mathbb{P}E^*,\mathcal{O}_{\mathbb{P}E^*}(1)\otimes \pi^* L^k)$
is Chow stable for $k \gg 0$. The goal of this article is to generalize this result  for the  polarizations  $(\mathbb{P}E^*,\mathcal{O}_{\mathbb{P}E^*}(d)\otimes \pi^* L^k)$  for positive integer $d$ and  $k \gg 0$. More precisely,

\begin{thm}\label{thm1}

Suppose that $Aut(X)$ is discrete and $X$ admits a constant scalar
curvature K\"ahler metric in the class of $2\pi  c_{1}(L)$. Let $d$ be a positive integer. If $E$
is Mumford stable, then 
$$(\mathbb{P}E^*,\mathcal{O}_{\mathbb{P}E^*}(d)\otimes \pi^* L^k)$$
is Chow stable for $k \gg 0$.

\end{thm}

The proof of the Theorem for general $d$ as opposed to $d=1$ needs a new result for the asymptotic expansion of the Bergman kernel.

In order to prove Thm.\ \ref{thm1} we use the concept of
\emph{balanced metrics}. Combining
the results of Luo, Phong, Sturm and Zhang (\cite{L}, \cite{PS1}, \cite{Zh}) on the relation between
balanced metrics and stability, it suffices to prove the following

\begin{thm}\label{thm2}
Let $X$ be a compact complex manifold and $L \rightarrow X$ be an
ample line bundle. Suppose that $X$ admits a constant scalar
curvature K\"ahler metric in the class of $2\pi  c_{1}(L)$ and
$Aut(X)$ is discrete. Let $E\rightarrow X$ be a holomorphic vector
bundle on $X$. If $E $ is Mumford stable, then
$\mathcal{O}_{\mathbb{P}E^*}(d)\otimes \pi^* L^k$ admits balanced 
metrics for $k \gg 0$.

\end{thm}
The balanced condition may be formulated in terms of Bergman
kernels. First, we show that there exists an asymptotic expansion
for the Bergman kernel of
$(\mathbb{P}E^*,\mathcal{O}_{\mathbb{P}E^*}(d)\otimes \pi^* L^k)$.
Fix a K\"ahler form $\omega$ on $X$ and a positive hermitian metric $\sigma$ on $L$ such that $i\bar{\partial}\partial \log
\sigma=\omega$. For any positive hermitian metric $g$ on
$\mathcal{O}_{\mathbb{P}E^*}(d)$, we define the sequence of volume
forms $d\mu_{g,k}$ on $\mathbb{P}E^*$ as follows
$$d\mu_{g,k}=k^{-m}\frac{(\omega_{g} + k
\pi^*\omega)^{m+r-1}}{(m+r-1)!}= \sum_{j=0}^{m}
 k^{j-m} \frac{\omega_{g}^{m+r-1-j}}{(m+r-j)!} \wedge
 \frac{\pi^*\omega^j}{j!},$$ where $\omega_{g}=i\bar{\partial}\partial \log g$.

Let $\rho_{k}(g,\omega)$ be the Bergman kernel of
$H^{0}(\mathbb{P}E^*,\mathcal{O}_{\mathbb{P}E^*}(d)\otimes
\pi^*L^{k})$ with respect to the $L^2$-inner product $L^2(g
\otimes \sigma^{\otimes k}, d\mu_{k,g})$. We prove the
 following theorem.

\begin{thm}{\label{thmH1}}
For any hermitian metric $h$ on $E$ and K\"ahler form $\omega \in
2\pi c_{1}(L)$, there exist smooth endomorphisms
$\widetilde{B}_{k}(h,\omega)$ such that $$\rho_{k}(g,\omega)([v])=
C_{r, d}^{-1}tr \big( \lambda_{d}(v,\textrm{Sym}^d h)\widetilde{B}_{k}(h,\omega)
\big),$$ where $g$ is the Fubini-Study metric on
$\mathcal{O}_{\mathbb{P}E^*}(d)$ induced by the hermitian metric
$h$, $C_{r,d}$ is a constant  defined by \eqref{eqN1} and  $ \lambda_{d}(v,\textrm{Sym}^d h)$ is an endomorphism of $\textrm{Sym}^d E$ defined in Def. \ref{defN1}. Moreover,
\begin{enumerate}
\item There exist smooth endomorphisms  $A_{i}(h,\omega) \in
\Gamma(X,End(\textrm{Sym}^d E))$ such that the following asymptotic expansion holds as
$k \longrightarrow \infty$,
$$ \widetilde{B}_{k}(h,\omega) \sim
k^m+A_{1}(h,\omega)k^{m-1}+\dots.$$

\item In particular
$$A_{1}(h,\omega)=\frac{r}{(r+d)}  \Big (  \Lambda F_{\textrm{Sym}^d h} - \frac{1}{R} tr( \Lambda F_{\textrm{Sym}^d h}) I_{\textrm{Sym}^d E}   \Big )+ \frac{1}{2} S(\omega) I_{\textrm{Sym}^d  E},$$ where $\Lambda $ is the trace operator acting
on $(1,1)$-forms with respect to the K\"ahler form $\omega$, 
$F_{\textrm{Sym}^d h}$ is the curvature of $(\textrm{Sym}^d E,\textrm{Sym}^d h)$, $R$ is the rank of the bundle $\textrm{Sym}^d E$ and $S(\omega)$ is the scalar curvature of $\omega$.

\item The asymptotic expansion holds in $C^{\infty}$. More
precisely, for any positive integers $a$ and $p$, there exists a
positive constant $K_{a,p,\omega,h}$ such that
$$\Big |\Big|\widetilde{B}_{k}(h,\omega)-\big(
k^m+\dots+A_{p}(h,\omega)k^{m-p} \big)\Big|\Big|_{C^a}\leq
K_{a,p,\omega,h} k^{m-p-1}.$$ Moreover the expansion is uniform in
the sense that there exists a positive integer $s$ such that if
$h$ and $\omega$ run in a bounded family in $C^s$ topology and
$\omega$ is bounded from below, then the constants $K_{a,p,\omega,h}$
are bounded by a constant depending only on $a$ and $p$.

\end{enumerate}

\end{thm}

Finding balanced metrics on $\mathcal{O}_{\mathbb{P}E^*}(d)\otimes
\pi^* L^k$ is basically the same as finding solutions to the
equations $\rho_{k}(g,\omega)= \textrm{Constant}.$ Therefore in
order to prove Thm.\ \ref{thm2}, we need to solve the equations
$\rho_{k}(g,\omega)= \textrm{Constant}$ for $k \gg 0$. Now if $h$ satisfies the
Hermitian-Einstein equation $\Lambda_{\omega}F_{(E,h)}=\mu I_{E}$, then  $\textrm{Sym}^d h$ satisfies a Hermitian-Einstein equation as well. Therefore if
$\omega$ has constant scalar curvature and  $h$ satisfies the
Hermitian-Einstein equation, then $A_{1}(h,\omega)$ is constant. Notice that in order to make
$A_{1}$ constant, existence of Hermitian-Einstein metric is not
enough. We need the existence of constant scalar curvature
K\"ahler metric as well.  Therefore if we know that the linearization of $A_{1}$ at $(h_{HE}, \omega_{CSK})$ is surjective, we could construct formal solutions as power series in 
$k^{-1}$ for the equation $\rho_{k}(g,\omega)= \textrm{Constant}.$ Unfortunately the linearization of $A_{1}$ at $(h_{HE}, \omega_{CSK})$ is  only onto the subspace of hermitian endomorphisms  of $\textrm{Sym}^d  E$ that are induced from endomorphisms of $E$. In order to overcome this issue, we generalize Theorem 1.3 to metrics of the form $\textrm{Sym}^d  h (I+k^{-1} \Phi)$, where $h$ is a metric on $E$ and $\Phi$ is a hermitian endomorphism of $\textrm{Sym}^d  E$ .
Let $g$ and $g_{k}(\Phi)$ be the Fubini-Study metrics on $\mathcal{O}_{\mathbb{P}E^*}(d)$ induced by $\textrm{Sym}^d  h $ and $\textrm{Sym}^d  h (I+k^{-1} \Phi)$ respectively. Let $\rho_{k}(g, \omega, \Phi)$ be the Bergman kernel of $(\mathcal{O}_{\mathbb{P}E^*}(d)\otimes \pi^* L^k)$ with respect to $L^2(g_{k}(\Phi)
\otimes \sigma^{\otimes k}, d\mu_{k,g, \Phi})$, where  $\displaystyle d\mu_{k,g, \Phi}= k^{-m}\frac{(\omega_{g_{k}(\Phi)}+ k \omega)^{m+r-1}}{(m+r-1)!}$ and $\omega_{g_{k}(\Phi)}= i\bar{\partial}\partial  \log g_{k}(\Phi)$.

\begin{thm}{\label{thmH1a}}
Let $h$ be a hermitian metric on $E$, $\Phi$ be a hermitian endomorphism of $\textrm{Sym}^d  E$  and $\omega \in
2\pi c_{1}(L)$ be a K\"ahler form. Then there exist smooth endomorphisms
$\widetilde{B}_{k}(h,\omega, \Phi)$ such that $$\rho_{k}(g,\omega, \Phi)([v])=
C_{r, d}^{-1}tr \big( \lambda_{d}(v,\textrm{Sym}^d h(I+k^{-1}\Phi))\widetilde{B}_{k}(h,\omega, \Phi)
\big),$$ where $g$ is the Fubini-Study metric on
$\mathcal{O}_{\mathbb{P}E^*}(d)$ induced by the hermitian metric
$h$. Moreover,
\begin{enumerate}
\item There exist smooth endomorphisms  $A_{i}(h,\omega. \Phi) \in
\Gamma(X,End(\textrm{Sym}^d E))$ such that the following asymptotic expansion holds as
$k \longrightarrow \infty$,
$$ \widetilde{B}_{k}(h,\omega, \Phi) \sim
k^m+A_{1}(h,\omega, \Phi)k^{m-1}+\dots.$$

\item In particular
$$A_{1}(h,\omega,\Phi)=A_{1}(h, \omega)+ \Phi-T(\Phi)$$ where $A_{1}(h,\omega)$ is given in Theorem \ref{thmH1}   and $T:\textrm{End}(\textrm{Sym}^d  E)\rightarrow \textrm{End}(\textrm{Sym}^d E)$ is a bundle map defined in Def. \ref{defN2}. 
\item The asymptotic expansion holds in $C^{\infty}$. More
precisely, for any positive integers $a$ and $p$, there exists a
positive constant $K_{a,p,\omega,h}$ such that
$$\Big |\big|\widetilde{B}_{k}(h,\omega, \Phi)-\big(
k^m+\dots+A_{p}(h,\omega, \Phi)k^{m-p} \big)\big|\Big|_{C^a}\leq
K_{a,p,\omega,h, \Phi} k^{m-p-1}.$$ Moreover the expansion is uniform in
the sense that there exists a positive integer $s$ such that if
$h$, $\omega$ and $\Phi$ run in a bounded family in $C^s$ topology and
$\omega$ and $h$ are bounded from below, then the constants $K_{a,p,\omega,h, \Phi}$
are bounded by a constant depending only on $a$ and $p$.

\end{enumerate}

\end{thm}

Next, the crucial fact is that the
linearization of $A_{1}$ at $(h,\omega, I_{\textrm{Sym}^d E})$ is surjective. This
enables us to construct formal solutions as power series in
$k^{-1}$ for the equation $\rho_{k}(g,\omega)= \textrm{Constant}.$
Therefore, for any positive integer $q$, we can construct a
sequence of metrics $g_{k}$ on
$\mathcal{O}_{\mathbb{P}E^*}(d)\otimes \pi^*L^{k}$ and bases
$s^{(k)}_{1},...,s_{N}^{(k)}$ for
$H^0(\mathbb{P}E^*,\mathcal{O}_{\mathbb{P}E^*}(d))$ such that
$$\sum |s^{(k)}_{i}|_{g_{k}}^2=1\ , \text{ and}\quad
\int \langle s^{(k)}_{i},s^{(k)}_{j}\rangle_{g_{k}} dvol_{g_{k}}=
D_{k}I+M_{k},$$ where $D_{k} \rightarrow C_{r,d}$ as $k \rightarrow
\infty$ (See \eqref{eqN1} for definition of $C_{r, d}$.), and $M_{k}$
is a trace-free hermitian matrix such that
$||M_{k}||_{\textrm{op}}=o(k^{-q-1})$ as $k\rightarrow \infty$.
Now  \cite[Theorem 4.6.]{S} implies that we can perturb these almost balanced metrics to get
balanced metrics. 

This article covers the following. In section $2$, we review some basics about symmetric powers. In section $3$, we prove the existence of an asymptotic expansions of Bergman kernels for $(\mathbb{P}E^*,\mathcal{O}_{\mathbb{P}E^*}(d)\otimes
\pi^*L^{k})$. We prove Theorem \ref{thmH1} and Theorem \ref{thmH1a} in this section. Section $4$ is devoted to construction of sequences of almost balanced metrics on $(\mathbb{P}E^*,\mathcal{O}_{\mathbb{P}E^*}(d)\otimes
\pi^*L^{k})$ using Theorem \ref{thmH1a}. In section $5$, we prove Theorem \ref{2011thm 1} and Theorem \ref{2011thm 1a} which guarantee that the asymptotic expansions obtained in Theorem \ref{thmH1} and Theorem \ref{thmH1a} hold in $C^{\infty}$. 
 
We refer the reader to \cite{S} for a history on the subject and related results.

\thanks{\textbf{Acknowledgements:} I am sincerely grateful to Richard Wentworth for  many helpful discussions and suggestions
and his continuous help, support and encouragement. I would  also like to thank Zhiqin Lu for many helpful discussions and suggestions.}

\section{Preliminaries}

Let $V$ be a complex vector space of dimension $r$. We define $$\textrm{Sym}^d V =V^{\otimes d}/\sim,$$ where
$$v_{1} \otimes \dots \otimes v_{d}\sim v_{\sigma(1)} \otimes \dots \otimes v_{\sigma(d)}$$ for any $\sigma \in S_{d}.$
We simply denote the class of $v_{1} \otimes \dots \otimes v_{d}$ in $\textrm{Sym}^d V$ by $v_{1}  \dots  v_{d}.$

Any hermitian inner product $h$ on $V$ defines a hermitian inner product  $\textrm{Sym}^d h$ on $\textrm{Sym}^d V$ by 
$$<v_{1}  \dots  v_{d}, w_{1}  \dots  w_{d}>_{\textrm{Sym}^d h}= \frac{1}{d!}\sum _{\sigma \in S_{d}} <v_{1},w_{\sigma(1)}> \dots <v_{d},w_{\sigma(d)}> .$$  

\begin{lem}

Let $e_{1}, \dots , e_{r}$ be a basis for $V$, then $$\{ e_{1}^{i_{1}} \dots e_{r}^{i_{r}} | 0\leq i_{\alpha} \leq d,\sum_{\alpha=1}^r j_{\alpha}=d \}$$ forms a basis for $\textrm{Sym}^d V$. Moreover, if the basis  $e_{1}, \dots , e_{r}$ is an orthonormal basis with respect to $h$, then    $$<e_{1}^{i_{1}} \dots e_{r}^{i_{r}}, e_{1}^{j_{1}} \dots e_{r}^{i_{r}}>_{\textrm{Sym}^d h}=0  \,\,\,\,\,\,  \textrm{if} \,\,\,\,\,\,\,\, (i_{1},\dots ,i_{r})\neq (j_{1},\dots ,j_{r}),$$
$$||e_{1}^{i_{1}} \dots e_{r}^{i_{r}}||_{\textrm{Sym}^d h}^2= \frac{i_{1}! \dots i_{r}!}{d!},$$ where $i_{1},\dots, i_{r}, j_{1}, \dots, j_{r}$ are integers such that $0\leq i_{\alpha},j_{\alpha} \leq d  $ and 
$\sum_{\alpha=1}^r i_{\alpha}=\sum_{\alpha=1}^r j_{\alpha}=d.$ 

\end{lem}

\begin{Def}{\label{defN1}}
 For any hermitian inner product $H$ on $\textrm{Sym}^d V$ and $v \in V$, define an endomorphism  $\lambda_{d}(v, H )$ of  $\textrm{Sym}^d V$ by $$\lambda_{d}(v, H)= |v_{d}|_{H}^{-2}v_{d} \otimes v_{d}^{*_{H}},$$
where $v_{d}= v \dots v.$  
For any  hermitian inner product $H$ on $\textrm{Sym}^d V$ and $v^* \in V^*$, define an endomorphism  $\lambda_{d}(v^*, H )$ of  $\textrm{Sym}^d V$ by $$\lambda_{d}(v^*, H)= |w|_{H}^{-2}w \otimes w^{*_{H}},$$
where $w$ is the unique vector in $\textrm{Sym}^d V$ satisfying $$v_{d}^*(u)=\langle u, w \rangle_{H}   \,\,\,\,\,\, \forall  u \in \textrm{Sym}^d V.$$

\end{Def}

There is a natural isomorphism between $\textrm{Sym}^d V$ and
$H^{0}(\mathbb{P}V^*,\mathcal{O}_{\mathbb{P}V^*}(d))$ which sends
$v_{1}\dots v_{d}\in \textrm{Sym}^d V$ to $\widehat{v_{1}\dots v_{d} }\in
H^{0}(\mathbb{P}V^*,\mathcal{O}_{\mathbb{P}V^*}(d)) $ defined by 
$$\widehat{v_{1}\dots v_{d} }([v^*])(w_{1}^* \otimes \dots \otimes w_{d}^*)= w_{1}^*(v_{1})\dots w_{d}^*(v_{d}),$$ where $w_{1}^* \in V^*$ and 
there exist complex numbers $\lambda_{1}, \dots ,\lambda_{d}$ such that $w_{i}^*=\lambda_{i} v^*. $
Notice that $$\widehat{v_{1}\dots v_{d} }([v^*])(w_{1}^* \otimes \dots \otimes w_{d}^*)=\lambda_{1} \dots \lambda_{d} v^*(v_{1})\dots v^*(v_{d})$$ and therefore it is a well-defined section of  $\mathcal{O}_{\mathbb{P}V^*}(d) $.

For any hermitian inner product $H$ on $\textrm{Sym}^d V$, define a metric $\widehat{H}$ on $\mathcal{O}_{\mathbb{P}V^*}(d) $ by 
$$\langle\hat{s},\hat{t}\rangle _{\widehat{H}}[v]= \frac{v^{\otimes d}(s) \overline{v^{\otimes d}(t)}}{|v \dots v|_{H}^{2}}.$$ Note that originally $H$ is a hermitian inner product on $\textrm{Sym}^d V$ and $v\dots v \in \textrm{Sym}^d V^*$. However, the hermitian inner product $H$ induces a hermitian inner product on $\textrm{Sym}^d V^*$ which we denote it by $H$ as well.
In particular $$\langle\hat{s},\hat{t}\rangle_{\widehat{\textrm{Sym}^d h}}[v]= \frac{v^{\otimes d}(s) \overline{v^{\otimes d}(t)}}{|v |_{h}^{2d}}.$$
The following lemma is straightforward.

\begin{lem}

For any hermitian inner product $h$ on $V$, we have  $$\hat{h}^{\otimes d}=\widehat{\textrm{Sym}^d h}.$$
\end{lem}

\begin{lem}{\label{lemN1} }
 There exists a constant $C_{r,d}$ such that for any $v,w \in \textrm{Sym}^d V$ and any hermitian inner product $h$ on $V$,
 
 $$d^{r-1}\int_{\mathbb{P}V^*} \langle \hat{v},\hat{w}\rangle_{\widehat{h}^{\otimes d}}
\frac{\omega_{\hat{h}}^{r-1}}{(r-1)!}= C_{r,d} \langle v,w \rangle _{\textrm{Sym}^d h},$$ where $\omega_{\hat{h}}= i\bar{\partial}\partial \log \hat{h}$.
The constant $C_{r,d}$ is given by the following formula
\begin{equation}\label{eqN1}C_{r,d}=\int_{\mathbb{C}^{r-1}}
\frac{d\xi \wedge d\overline{\xi}}{(1+\sum_{j=1}^{r-1}
|\xi_{j}|^{2})^{r+d}}.\end{equation} Here $d\xi \wedge
d\overline{\xi}= (\sqrt{-1}d\xi_{1} \wedge
d\overline{\xi}_{1})\wedge \dots \wedge (\sqrt{-1}d\xi_{r-1}
\wedge d\overline{\xi}_{r-1}).$

Conversely, let $H$ be a hermitian inner product on $\textrm{Sym}^d V$. Suppose there exists a constant $C$ such that 
 $$\int_{\mathbb{P}V^*} \langle \hat{v},\hat{w}\rangle_{\widehat{H}} \frac{\omega_{\widehat{H}}^{r-1}}{(r-1)!}= C \langle v,w \rangle _{H},$$  for any $v,w \in \textrm{Sym}^d V$. Then there exists a hermitian inner product $h$ on $V$ such that $H=\textrm{Sym}^d h$.

\end{lem}

\begin{proof}
The first part is a straightforward computation. For the second part, suppose that $H$ is a hermitian inner product on $\textrm{Sym}^d V$ satisfying $$\int_{\mathbb{P}V^*} \langle \hat{v},\hat{w}\rangle_{\widehat{H}} \frac{\omega_{\widehat{H}}^{r-1}}{(r-1)!}= C \langle v,w \rangle _{H},$$  for any $v,w \in \textrm{Sym}^d V$. 
Let $v_{1}, \dots v_{R}$ be an orthonormal basis for $\textrm{Sym}^d V$ with respect to $H$.
For any $e^* \in V^*$, we have $$\sum||\widehat{v_{i}}||^2_{\widehat{H}}([e^*])= \frac{\sum e^*_{d}(v_{i})\overline{e^*_{d}(v_{i})}}{||e^*_{d}||^2_{H}}=\frac{||e^*_{d}||^2_{H}}{||e^*_{d}||^2_{H}}=1.$$
On the other hand $$\int_{\mathbb{P}V^*} \langle \hat{v_{i}},\hat{v_{j}}\rangle_{\widehat{H}} \frac{\omega_{\widehat{H}}^{r-1}}{(r-1)!}= C \langle v_{i},v_{j} \rangle _{H}=C\delta_{ij}.$$ 
Therefore $\widehat{H}$ is a balanced metric on $(\mathbb{P}V^*,\mathcal{O}_{\mathbb{P}V^*}(d))$. It concludes the proof since balanced metrics on $(\mathbb{P}V^*,\mathcal{O}_{\mathbb{P}V^*}(d))$ are unique up to $Aut(\mathbb{P}V^*,\mathcal{O}_{\mathbb{P}V^*}(d))\cong PGL(V)$.

\end{proof}

There is a canonical representation of $\textrm{Sym}^d :GL(V) \rightarrow GL(\textrm{Sym}^d  V )$ defined as follows:
$$(\textrm{Sym}^d A)(v_{1}\dots v_{d})= Av_{1}\dots  Av_{d},$$ where $A \in GL(V)$ and $v_{1}, \dots , v_{d} \in V$. This induces a Lie algebra homomorphism $S^d: End(V) \rightarrow End(\textrm{Sym}^d V)$ defined by \begin{equation}\label{2011eq9}(S^dA)(v_{1}\dots v_{d})= \sum_{i=1}^d v_{1} \dots Av_{i} \dots v_{d}\end{equation} for any $A \in End(V).$ Suppose that the vector space $V$ is equipped with a hermitian inner product. Then the Lie algebra homomorphism $S^d$ maps hermitian endomorphisms to hermitian endomorphisms. More precisely, we have the following.

\begin{lem}

Let $h$ be a hermitian inner product on $V$. We denote the space of hermitian endomorphisms  of $V$ with respect to $h$ by $End_{h}(V)$ and the space of hermitian endomorphisms  of $\textrm{Sym}^d V$ with respect to $\textrm{Sym}^d h$ by $End_{h}(\textrm{Sym}^d V)$.   
Then $$S^d(End_{h}(V)) \subset End_{h}(\textrm{Sym}^d V).$$

\end{lem}

Let $E$ be a holomorphic vector bundle over a K\"ahler manifold $(X,\omega)$ and $h$ be a hermitian metric on $E$. Then straightforward computation shows that 
 $$F_{(\bar{\partial}_{\textrm{Sym}^d E}, \textrm{Sym}^d h)}=S^d F_{(\bar{\partial}_{E},h)},$$ where $F_{(\bar{\partial}_{E},h)}$ is the curvature of the chern connection on $(E,h)$.
A direct consequence of the above formula is the following:

\begin{prop}

Let $h_{HE}$ be a Hermitian-Einstein metric on $E$ with respect to $\omega$, i.e. $\Lambda_{\omega} F_{h_{HE}}=\mu I_{E}$. Then $\textrm{Sym}^d h_{HE}$
is a Hermitian-Einstein metric on $\textrm{Sym}^d E$ with respect to $\omega$. 

\end{prop}

\section{Asymptotic Expansion}

The goal of this section is to give an asymptotic expansion for the Bergman kernel
of $(\mathbb{P}E^*,\mathcal{O}_{\mathbb{P}E^*}(d)\otimes
\pi^*L^{k})$.

Let $(X,\omega)$ be a K\"ahler manifold of dimension $m$ and $E$
be a holomorphic vector bundle on $X$ of rank $r$. Let $L$ be an
ample line bundle on $X$ endowed with a hermitian metric $\sigma$
such that $i\bar{\partial}\partial \log \sigma=\omega$. For any hermitian metric $h$ on
$E$ , we define the volume form
$$d\mu_{g}=\frac{\omega_{g}^{r-1}}{(r-1)!}\wedge \frac{\pi^*
\omega^m}{m!},$$ where $g=\widehat{\textrm{Sym}^d h}=\widehat{h}^{\otimes d},$
 $\omega_{g}= i\bar{\partial}\partial \log g=d i\bar{\partial}\partial \log\widehat{h}$ and
$\pi:\mathbb{P}E^*\rightarrow X$ is the projection map. The goal
is to find an asymptotic expansion for the Bergman kernel of
$\mathcal{O}_{\mathbb{P}E^*}(d) \otimes L^k \rightarrow
\mathbb{P}E^*$ with respect to the $L^2$-metric defined on
$H^{0}(\mathbb{P}E^*,\mathcal{O}_{\mathbb{P}E^*}(d)\otimes
\pi^*L^{k})$. We define the $L^2$- metric using the fibre metric
$g \otimes \sigma^{\otimes k}$ and the volume form $d\mu_{g,k}$
defined as follows
\begin{equation}\label{2011eq13}d\mu_{g,k}=k^{-m}\frac{(\omega_{g} + k
\omega)^{m+r-1}}{(m+r-1)!}= \sum_{j=0}^{m} k^{j-m}
\frac{\omega_{g}^{m+r-1-j}}{(m+r-j)!} \wedge
\frac{\omega^j}{j!}.\end{equation}

In order to do that, we reduce the problem to the problem of
Bergman kernel asymptotics on $\textrm{Sym}^d E\otimes L^k \rightarrow X$. The
first step is to use the volume form $d\mu_{g}$ which is a product
volume form instead of the more complicated one $d\mu_{g,k}$. So,
we replace the volume form $d\mu_{g,k}$ with $d\mu_{g}$ and the
fibre metric $g \otimes \sigma^k$ with $g(k) \otimes \sigma^k$,
where the metrics $g(k)$ are defined on
$\mathcal{O}_{\mathbb{P}E^*}(d)$ by
\begin{equation}\label{eqH2}g(k)=k^{-m}(\sum_{j=0}^{m} k^j
f_{j})g=(f_{m}+k^{-1}f_{m-1}+...+k^{-m}f_{0})g,\end{equation} and
\begin{equation} \label{eqH3}\frac{\omega_{g}^{m+r-1-j}}{(m+r-j)!} \wedge
\frac{\omega^j}{j!}=f_{j}d\mu_{g}. \end{equation} Clearly the
$L^2$-inner products $L^2(g \otimes \sigma^k,d\mu_{g,k})$ and
$L^2(g(k) \otimes \sigma^k,d\mu_{g})$ on
$H^{0}(\mathbb{P}E^*,\mathcal{O}_{\mathbb{P}E^*}(d)\otimes
\pi^*L^{k})$ are the same. The second step is going from
$\mathcal{O}_{\mathbb{P}E^*}(d) \rightarrow \mathbb{P}E^*$ to
$\textrm{Sym}^d E \rightarrow X$. In order to do this we somehow push forward the
metric $g(k)$ to get a metric $\widetilde{g}(k)$ on $\textrm{Sym}^d  E$ (See
Definition \ref{defH4}). Then we can apply the result on the
asymptotics of the Bergman kernel on $\textrm{Sym}^d E$. The last step is to use
this to get the result.

\begin{Def}
Let $\widehat{s_{1}^k},....,\widehat{s_{N}^k} $ be an orthonormal
basis for
$H^{0}(\mathbb{P}E^*,\mathcal{O}_{\mathbb{P}E^*}(d)\otimes
\pi^*L^{k})$ with respect to $L^2(g \otimes \sigma^k, d\mu_{k,g})$. We
define
\begin{equation}\label{eqH4}\rho_{k}(g,\omega)=\sum_{i=1}^{N}
|\widehat{s_{i}^k}|_{g\otimes \sigma^k}^2.\end{equation}
\end{Def}

\begin{Def}\label{defH4}
For any hermitian form $g$ on $\mathcal{O}_{\mathbb{P}E^*}(d)$, we
define a hermitian form $\widetilde{g}$ on $\textrm{Sym}^d E$ as follow
\begin{equation}\label{eqH6}\widetilde{g}(s,t)=C_{r,d}^{-1} \int_{\mathbb{P}E_{x}^*}
g(\widehat{s},\widehat{t}\, )
\frac{\omega_{g}^{r-1}}{(r-1)!},\end{equation} for $s,t \in
\textrm{Sym}^d E_{x}.$ (See \eqref{eqN1} for definition of $C_{r,d}$.)
\end{Def}

Notice that if $g=\widehat{\textrm{Sym}^d h}$ for some hermitian metric $h$ on
$E$, Lemma \ref{lemN1} implies that $\widetilde{g}=\textrm{Sym}^d h.$
Define hermitian metrics $\widetilde{g_{j}}$'s on $\textrm{Sym}^d E$ by
\begin{equation}\label{eqH8a}\widetilde{g_{j}}(s,t)=C_{r, d}^{-1} \int_{\mathbb{P}E_{x}^*}
 f_{j}g(\widehat{s},\widehat{t} ) \frac{\omega_{g}^{r-1}}{(r-1)!},\end{equation} for
$s,t \in \textrm{Sym}^d E_{x}.$ Also we define $\Psi_{j} \in End(\textrm{Sym}^d E)$ by
\begin{equation}\label{eqH8}\widetilde{g_{j}}= \textrm{Sym}^d h \, \Psi_{j}.\end{equation}

If $h$ and $\omega$ vary in bounded family, then $\Psi_{j}$'s vary in a bounded family. More precisely, we have the following

\begin{thm}\label{2011thm 1}
Let $\nu_{0}$ be a fixed K\"ahler form on $X$  and $h_{0}$ be a fixed hermitian metric on $E$. For any positive numbers $l$ and $l'$ and any
positive integer $p$, there exists a positive number $C_{l,l',p}$
such that if $$||\omega||_{C^p(\nu_{0})}, ||h||_{C^{p+2}(h_{0},\nu_{0})}
\leq l \qquad\text{and}\qquad \inf_{x \in X}|\omega(x)^m|_{\nu_{0}(x)}\geq l',$$
then
$||\Psi_{i}||_{C^p(h_{0}, \nu_{0})} \leq C_{l,l',p}$, for any $1 \leq i\leq m$.
\end{thm}

We prove Theorem \ref{2011thm 1} in Section $5$.

\begin{lem}\label{lemH2}
We have the following
\begin{enumerate}
\item $\displaystyle \Psi_{m}=I_{{Sym}^d E}.$ \item $\displaystyle
\Psi_{m-1}=\frac{d}{(r+d)} \Big(  \Lambda F_{{Sym}^d h}+ tr( \Lambda F_{h})  I_{{Sym}^d E} \Big)$.

\end{enumerate}

\end{lem}

\begin{proof}

Fix a point $p \in X$. Let $e_{1},...,e_{r}$ be a local
holomorphic frame for $E$ around $p$ such that $$\langle e_{i},
e_{j}\rangle _{h}(p)=\delta_{ij}, \,\,\,\,\,\,\,\, d \langle
e_{i}, e_{j}\rangle _{h}(p)=0 .$$ For simplicity, we assume that $$\frac{i}{2\pi
}F_{h}(p)=\left(
\begin{matrix}
 \omega_{1} & 0& \cdots &0 \\
 0 & \omega_{2}& \cdots & 0\\
 \vdots &  & \ddots  &\vdots \\
 0 & 0  & \cdots & \omega_{r} \end{matrix} \right).$$

Let $\lambda_{1},...,\lambda_{r}$ be the homogeneous coordinates
on the fibre. At the fixed point $p$, we have $$\omega_{g}=
d(\omega_{\textrm{FS},h}+\frac{\sum \omega_{i} |\lambda_{i}|^2}{\sum
|\lambda_{i}|^2}).$$

Therefore,\begin{align*}\omega_{g}^{r} \wedge
\omega^{m-1}&=d \omega_{\textrm{FS},h}^{r-1}
\wedge \big(\frac{\sum \omega_{i} |\lambda_{i}|^2}{\sum
|\lambda_{i}|^2}\big) \wedge \omega^{m-1}.\end{align*}

Hence

$$f_{m-1}= d\frac{\sum |\lambda_{i}|^2\Lambda \omega_{i}}{\sum |\lambda_{i}|^2}. $$

Let $\alpha_{1}, \dots, \alpha_{r}$ be nonnegative integers such that $\alpha_{1}+ \dots +\alpha_{r}=d$. 
Therefore,

\begin{align*}\widetilde{g_{m-1}}(e_{1}^{\alpha_{1}}\dots e_{r}^{\alpha_{r}}, e_{1}^{\alpha_{1}}\dots & e_{r}^{\alpha_{r}})=C_{r,d}^{-1}\pi_{*}(f_{m-1} \tilde{g}(\widehat{e_{1}^{\alpha_{1}}\dots e_{r}^{\alpha_{r}}}, \widehat{e_{1}^{\alpha_{1}}\dots e_{r}^{\alpha_{r}}})\frac{\omega_{g}^{r-1}}{(r-1)!})\\&=C_{r,d}^{-1} C_{r} \sum \Lambda \omega_{i} \int_{\mathbb{C}^{r-1}}\frac{|\lambda_{i}|^2|\lambda_{1}|^{2\alpha_{1}}\dots |\lambda_{r}|^{2\alpha_{r}}d\lambda \wedge d\overline{\lambda}}{(1+\sum_{j=1}^{r-1}|\lambda_{j}|^{2})^{r+d+1}}\\&= C_{r,d}^{-1} C_{r}\frac{r! \alpha_{1}! \dots \alpha_{r}! }{(r+d)!}\sum_{i=1}^r (\alpha_{i}+1) \Lambda \omega_{i}.\end{align*}

Hence, $$\Psi_{m-1}=\frac{ d}{(r+d)} \Big(  \Lambda F_{{Sym}^d h}+ tr( \Lambda F_{h})  I_{{Sym}^d E} \Big)$$

\end{proof}

The following lemmas are straightforward.

\begin{lem}\label{lemH3}

$\displaystyle \widetilde{g \otimes \sigma^k}=\widetilde{g}
\otimes \sigma^k.$

\end{lem}

\begin{lem}\label{lemH4} Let $H$ be a hermitian metric on $\textrm{Sym}^d E$ and $s_{1},..., s_{N}$ be a basis for $H^0(X,{Sym}^d E)$. Then
$$\sum |\widehat{s_{i}}([v^*])|^2_{\widehat{H}}= Tr \big( B \lambda_{d}(v^*,H)          \big),$$
where $B=\sum s_{i} \otimes s_{i}^{*_{H}}.$

\end{lem}

\begin{proof}[Proof of Theorem  \ref{thmH1}]
Define \begin{equation}\label{2011eq3}h(k)=C_{r,d}^{-1}\pi_{*} \big(g\frac{d\mu_{g,k}}{d\mu_{g}}\frac{\omega_{g}^{r-1}}{(r-1)!} \big ),\end{equation} i.e. for any $x \in X$ and $s, t \in \textrm{Sym}^d E_{x}$, we have 
 $$\langle s, t \rangle_{h(k)}=C_{r,d}^{-1}\int_{\mathbb{P}E_{x}^*} \langle \widehat{s},\widehat{t}\, \rangle_{g}\frac{d\mu_{g,k}}{d\mu_{g}} \frac{\omega_{g}^{r-1}}{(r-1)!}.$$
Therefore \eqref{2011eq13}, \eqref{eqH3} and \eqref{eqH8} imply that 
\begin{equation}\label{eqH9}h(k)=\sum_{j=0}^{m}
k^{j-m}\widetilde{g}_{j}=\textrm{Sym}^d h  (\sum_{j=0}^{m}
k^{j-m}\Psi_{j}).\end{equation} Let $B_{k}(h(k),\omega)$ be the
Bergman kernel of $\textrm{Sym}^d E\otimes L^k$ with respect to the $L^2$-metric
defined by the hermitian metric $h(k)\otimes \sigma^k$ on
$\textrm{Sym}^d E\otimes L^k$ and the volume form $\frac{ \omega^m}{m!}$ on $X$.
Therefore, if $s_{1},..., s_{N}$ is an orthonormal basis for
$H^0(X,\textrm{Sym}^d E\otimes L^k)$ with respect to the $L^2(h(k)\otimes
\sigma^k,\frac{ \omega^m}{m!})$, then
\begin{equation} B_{k}(h(k),\omega)=\sum s_{i}\otimes s_{i}^{*_{h(k)\otimes
\sigma^k}},\end{equation} We define $\widetilde{B}_{k}(h,\omega)$
as follows:
\begin{equation}\widetilde{B}_{k}(h,\omega)=\sum s_{i}\otimes s_{i}^{*_{\textrm{Sym}^d h\otimes
\sigma^k}}.\end{equation} Let
$\widehat{s_{1}},....,\widehat{s_{N}} $ be the corresponding basis
for $H^{0}(\mathbb{P}E^*,\mathcal{O}_{\mathbb{P}E^*}(d)\otimes
L^{k})$. Hence,
\begin{align*}\int_{\mathbb{P}E^*} \langle \widehat{s_{i}}, \widehat{s_{j}}
\rangle_{g \otimes \sigma^k}d\mu_{g,k}&=\int_{\mathbb{P}E^*}
\langle \widehat{s_{i}}, \widehat{s_{j}} \rangle_{g \otimes
\sigma^k}(\sum_{j=0}^{m} k^j f_{j})d\mu_{g}\\&= \int_{\mathbb{P}E^*} \langle \widehat{s_{i}}, \widehat{s_{j}}
\rangle_{g(k) \otimes \sigma^k}d\mu_{g}\\&=C_{r, d} \int_{X}\langle
s_{i}, s_{j} \rangle_{h(k) \otimes
\sigma^k}\frac{\omega^m}{m!}\\&=C_{r, d}\delta_{ij}.\end{align*} Therefore
$\frac{1}{\sqrt{C_{r,d}}}\widehat{s_{1}},....,\frac{1}{\sqrt{C_{r, d}}}\widehat{s_{N}}
$ is an orthonormal basis for
$H^{0}(\mathbb{P}E^*,\mathcal{O}_{\mathbb{P}E^*}(d)\otimes L^{k})$
with respect to $L^2(g \otimes \sigma^k, d\mu_{k,g})$.  Hence Lemma
\ref{lemH4} implies that
$$C_{r,d}\rho_{k}(g, \omega)=Tr \big(  \lambda_{d}(v^*, \textrm{Sym}^d h)
\widetilde{B}_{k}(h,\omega)       \big).$$ Now, in order to
conclude the proof, it suffices to show that there exist smooth
endomorphisms $A_{i} \in \Gamma(X,End(\textrm{Sym}^d E))$ such that
$$ \widetilde{B}_{k}(h,\omega)\sim k^m+A_{1}k^{m-1}+\dots              .$$
Let $B_{k}(\textrm{Sym}^d h,\omega)$ be the Bergman kernel of $\textrm{Sym}^d E\otimes L^k$ with
respect to the $L^2(\textrm{Sym}^d h\otimes \sigma^k, \frac{\omega^m}{m!})$. A fundamental result on
the asymptotics of the Bergman kernel (\cite{C}, \cite{Z}, \cite{Lu}, \cite{W}) states
that there exists an asymptotic expansion
$$B_{k}(\textrm{Sym}^d h,\omega)\sim k^m+B_{1}(\textrm{Sym}^d h)k^{m-1}+\dots,$$ where $$B_{1}(\textrm{Sym}^d h)=
\frac{i}{2\pi} \Lambda F_{(\textrm{Sym}^d E,\textrm{Sym}^d h)}+ \frac{1}{2} S(\omega) I_{\textrm{Sym}^d E}
.$$(See also \cite{BBS}.) Moreover this expansion holds
uniformly for any $h$ in a bounded family. Therefore, we can
Taylor expand the coefficients $B_{i}(\textrm{Sym}^d h)$'s. We conclude that for
endomorphisms $\Phi_{1},...,\Phi_{M}$,
$$B_{k}(\textrm{Sym}^d h(I+\sum_{i=0}^{M} k^{-i}\Phi_{i}),\omega)\sim
k^m+B_{1}(\textrm{Sym}^d h)k^{m-1}+\dots .$$Note that $B_{1}(\textrm{Sym}^d h)$ in the above
expansion does not depend on $\Phi_{i}$'s and is given as before
by $$B_{1}(\textrm{Sym}^d h)=
\frac{i}{2\pi} \Lambda F_{(\textrm{Sym}^d E,\textrm{Sym}^d h)}+ \frac{1}{2} S(\omega) I_{\textrm{Sym}^d E}.$$ On the other hand
\begin{align*}B_{k}(h(k),\omega)=\sum s_{i}\otimes
s_{i}^{*_{\widetilde{g(k)}\otimes \sigma^k}}&=(\sum s_{i}\otimes
s_{i}^{*_{h\otimes \sigma^k}})(\sum_{j=0}^m k^{j-m}\Psi_{j})\\&=\widetilde{B}_{k}(h,\omega)(\sum_{j=0}^m
k^{j-m}\Psi_{j}).\end{align*} Therefore,
\begin{align*}\widetilde{B}_{k}(\textrm{Sym}^d h,\omega)&=B_{k}(h(k),\omega)(\sum_{j=0}^m k^{j-m}\Psi_{j})^{-1}\\& \sim k^m+(B_{1}(\textrm{Sym}^d h)-\Psi_{m-1})k^{m-1}+\dots  .\end{align*} 
We have \begin{align}\label{2011eq12}B_{1}-\Psi_{m-1}&=\frac{ir}{2\pi(r+d)} \Lambda F_{\textrm{Sym}^d h}+ \frac{1}{2} S(\omega) I_{\textrm{Sym}^d E}-\frac{id}{2\pi(r+d)} tr( \Lambda F_{h})  I_{{Sym}^d E}\notag \\&=\frac{ir}{2\pi(r+d)}  \Big (  \Lambda F_{\textrm{Sym}^d h} - \frac{1}{\textrm{rank}\big(\textrm{Sym}^d E\big)} tr( \Lambda F_{\textrm{Sym}^d h} )I  \Big )+ \frac{1}{2} S(\omega)I \end{align}


Notice that Theorem \ref{2011thm 1} implies that if $h$ and $\omega$ vary in a bounded family and $\omega$ is bounded from below, then $\Psi_{1},..,\Psi_{m}$ vary
in a bounded family. Therefore the asymptotic expansion that we obtained for $\widetilde{B}_{k}(h,\omega)$ is uniform as long as $h$ and $\omega$ vary in a bounded family and $\omega$ is bounded from below.

\end{proof}

Suppose that $E$ admits a Hermitian-Einstein metric $h_{HE}$ and $(X, L)$ admits a constant scalar curvature K\"ahler metric $\omega_{CSCK}$.
If the linearization of $A_{1}$ at $(h_{HE}, \omega_{CSCK})$ were surjective, then we would be able to construct sequences of almost balanced metrics. 
The problem is that the image of the linearization of $A_{1}$ consists only those endomorphisms of $\textrm{Sym}^d E$ that are induced from endomorphisms of $E$. 
Therefore we need to generalize Theorem \ref {thmH1} .

Let $\Phi \in \Gamma(\textrm{End} (\textrm{Sym}^d E)) $ be hermitian with respect to  $\textrm{Sym}^d h$. As before, let $g$ be the Fubini-Study metric on $\mathcal{O}_{\mathbb{P}E^*}(d)$ induced by the hermitian metric $h$. Define hermitian metrics \begin{equation} \label{2011eq5}h_{t}(\Phi)=\textrm{Sym}^d h(I+t\Phi) \,\,\,\, \textrm{and} \,\,\, g_{k}(\Phi)=\widehat{h_{k^{-1}}(\Phi)}\end{equation} on $\textrm{Sym}^d E$ and $\mathcal{O}_{\mathbb{P}E^*}(d)$ respectively. We define the function $ F(\Phi) \in C^{\infty}(\mathbb{P}E^*)$ by $$F(\Phi)([v])=||v_{d}||_{\textrm{Sym}^d h}^{-2}\frac{d}{dt}\Big|_{t=0} ||v_{d}||_{h_{t}(\Phi)}^2,\,\,\,\,\,\, v \in E^*.$$ Here $v_{d}=v \dots v$ and note that $||v_{d}||_{\textrm{Sym}^d h}^2=||v||_{h}^{2d}.$
Simple calculations show that \begin{equation}\label{2011eq6} F(\Phi)([v])=tr(\lambda_{d}(v,\textrm{Sym}^d h)\Phi). \end{equation} \begin{equation} \label{2011eq7}\frac{g_{k}(\Phi)}{g}([v])=1+ k^{-1 } F(\Phi)+O(k^{-2}).\end{equation}
Thus, \begin{equation} \label{2011eq8}\omega_{g_{k}(\Phi)}=i\bar{\partial}\partial \log g_{k}(\Phi)=\omega_{ g}+  k^{-1 }i\bar{\partial}\partial F(\Phi)+O(k^{-2}).\end{equation}
We define the volume forms $d\mu_{g,\Phi,k}$ on $\mathbb{P}E^*$ as follows
\begin{equation}d\mu_{g,\Phi,k}=k^{-m}\frac{(\omega_{g_{k}(\Phi)} + k \omega)^{m+r-1}}{(m+r-1)!}.\end{equation}
For any smooth function $F \in C^{\infty}(\mathbb{P}E^*) $, define $$  \widetilde{\triangle}F =  \frac{(r-1) i\bar{\partial}\partial F \wedge \omega_{g}^{r-2} \wedge \omega^m}{ \omega_{g}^{r-1} \wedge \omega^m}.     $$
Therefore \eqref{eqH3},  \eqref{2011eq7} and \eqref{2011eq8} imply that 

\begin{equation} \label{2011eq10}g_{k}(\Phi)d\mu_{g,\Phi,k}=\Big(1+k^{-1} \big(    f_{m-1}+ F(\Phi)+ \widetilde{\triangle}F(\Phi)  \big)+ O(k^{-2})\Big) g d\mu_{g}    .\end{equation} Recall that $d\mu_{g}=\frac{\omega_{g}^{r-1}}{(r-1)!}\wedge \frac{\omega^{n}}{n!}.$


\begin{Def}{\label {defN2}}

Define the bundle map   $T: \textrm{End} (\textrm{Sym}^d E) \rightarrow \textrm{End} (\textrm{Sym}^d E)$ by 

\begin{equation}\label{2011eq11}\langle s, (T\Phi)(t) \rangle_{\textrm{Sym}^d h}= C_{r,d}^{-1} \int_{\mathbb{P}E_{x}^*} (  F(\Phi)+ \widetilde{\triangle}F(\Phi) )\langle \widehat{s},\widehat{t}\, \rangle_{g} \frac{\omega_{g}^{r-1}}{(r-1)!},\end{equation} for any $x \in X$ and $s, t \in \textrm{Sym}^d E_{x}$.

\end{Def}
We will use the following Lemmas in the proof of Corollary \ref{corH2}. 

\begin{lem}\label{2011lem4}
For any $\Phi \in \Gamma(\textrm{End} (\textrm{Sym}^d E)) $ hermitian with respect to  $\textrm{Sym}^d h$, we have $Tr(T\Phi)=Tr(\Phi).$
\end{lem}

\begin{proof}

Let $e_{1}, \dots e_{r}$ be an orthonormal local frame for $E$ with respect to $h$ and $E_{I}$ be the corresponding orthonormal local frame for $\textrm{Sym}^d E$ with respect to $\textrm{Sym}^d h$. We have 
\begin{align*}  Tr(T\Phi)= \sum_{I} \langle E_{I}, (T\Phi)(E_{I}) \rangle_{\textrm{Sym}^d h}&=\sum_{I} C_{r,d}^{-1}\int_{\textrm{Fiber}} (F+\widetilde{\triangle}F)\langle \widehat{E_{I}}, \widehat{E_{I}} \rangle_{g} \frac{\omega_{g}^{r-1}}{(r-1)!} \\&=C_{r,d}^{-1}\int_{\textrm{Fiber}} (F+\widetilde{\triangle}F)\sum_{I}|\widehat{E_{I}}|_{g}^2 \frac{\omega_{g}^{r-1}}{(r-1)!}\\&=C_{r,d}^{-1}\int_{\textrm{Fiber}} (F+\widetilde{\triangle}F)\frac{\omega_{g}^{r-1}}{(r-1)!}\\&=C_{r,d}^{-1}\int_{\textrm{Fiber}} F\frac{\omega_{g}^{r-1}}{(r-1)!}. \end{align*}
On the other hand, \eqref{2011eq6} implies that $$\int_{\textrm{Fiber}} F\frac{\omega_{g}^{r-1}}{(r-1)!}=C_{r,d} Tr(\Phi).$$

\end{proof}

\begin{lem}\label{2011lem3}

For any $\varphi \in \Gamma(\textrm{End} ( E))$, we have $T(S^d\varphi)= S^d \varphi $. Conversely,
if $T\Phi=\Phi$ for some $\Phi \in \Gamma(\textrm{End} (\textrm{Sym}^d E))$, then there exists $\varphi \in \Gamma(\textrm{End} ( E))$ such that $\Phi=S^d \varphi$. (See \eqref{2011eq9} for definition of Lie algebra homomorphim $S^d$.)

\end{lem} 

\begin{proof}

The equations \eqref{2011eq7} and \eqref{2011eq8} imply that $$g_{k}(\Phi)\omega_{g_{k}(\Phi)}^{r-1}=\Big(1+k^{-1} \big(    F(\Phi)+ \widetilde{\triangle}F(\Phi)  \big)+ O(k^{-2})\Big) g \omega_{g}^{r-1}.$$ Therefore 
\begin{align*}\int_{\mathbb{P}E_{x}^*} \langle \widehat{s},\widehat{t}\, \rangle_{g_{k}(\Phi)} \frac{\omega_{g_{k}(\Phi)}^{r-1}}{(r-1)!}&=  \int_{\mathbb{P}E_{x}^*} \Big(1+  \big(F(\Phi)+ \widetilde{\triangle}F(\Phi)\big)k^{-1}+O(k^{-2}) \Big)\langle \widehat{s},\widehat{t}\, \rangle_{g} \frac{\omega_{g}^{r-1}}{(r-1)!}\\&= \int_{\mathbb{P}E_{x}^*}\langle \widehat{s},\widehat{t}\, \rangle_{g} \frac{\omega_{g}^{r-1}}{(r-1)!}+k^{-1}\int_{\mathbb{P}E_{x}^*}   \big(F(\Phi)+ \widetilde{\triangle}F(\Phi)\big)\langle \widehat{s},\widehat{t}\, \rangle_{g} \frac{\omega_{g}^{r-1}}{(r-1)!}\\&+O(k^{-2}). \end{align*}
Lemma \ref{lemN1} and \eqref{2011eq11} imply that 
$$C_{r,d}^{-1}\int_{\mathbb{P}E_{x}^*} \langle \widehat{s},\widehat{t}\, \rangle_{g_{k}(\Phi)} \frac{\omega_{g_{k}(\Phi)}^{r-1}}{(r-1)!}= \langle s,t\rangle_{\textrm{Sym}^d h}+k^{-1}\langle s,(T\Phi)(t)\rangle_{\textrm{Sym}^d h}+O(k^{-2}),$$since $g=\widehat{\textrm{Sym}^d h}$. On the other hand, Lemma \ref{lemN1} implies that $$C_{r,d}^{-1}\int_{\mathbb{P}E_{x}^*} \langle \widehat{s},\widehat{t}\, \rangle_{g_{k}(\Phi)} \frac{\omega_{g_{k}(\Phi)}^{r-1}}{(r-1)!}=\langle s,t\rangle_{\textrm{Sym}^d h(I+k^{-1}\Phi)}$$ if and only if there exists  $\varphi \in \Gamma(\textrm{End} ( E))$ such that $\Phi=S^d \varphi$. This concludes the proof.

\end{proof}

\begin{Def}

 Let $h$ be a hermitian metric on $E$ and  $\Phi \in \Gamma(\textrm{End} (\textrm{Sym}^d E)) $ be hermitian with respect to  $\textrm{Sym}^d h$.
We define $$\rho_{k}(g, \omega, \Phi):= \rho_{k}( g_{k}(\Phi), \omega)=\rho_{k}( \widehat{\textrm{Sym}^d h(I+k^{-1}\Phi)}, \omega),$$ where $g=\widehat{\textrm{Sym}^d h}.$

\end{Def}

In order to prove Theorem \ref{thmH1a}, we need to find an asymptotic expansion for the Bergman kernel $\rho_{k}(g, \omega, \Phi):= \rho_{k}( g_{k}(\Phi), \omega).$ By definition $\rho_{k}(g, \omega, \Phi)$ is the Bergman kernel of
$\mathcal{O}_{\mathbb{P}E^*}(d) \otimes L^k \rightarrow
\mathbb{P}E^*$ with respect to the inner product $L^2(g_{k}(\Phi), d\mu_{g,k,\Phi} )$ defined on
$H^{0}(\mathbb{P}E^*,\mathcal{O}_{\mathbb{P}E^*}(d)\otimes
\pi^*L^{k})$. Clearly the $L^2-$ inner products $L^2(g_{k}(\Phi), d\mu_{g,k,\Phi} )$ and $L^2(g_{k}(\Phi)\frac{d\mu_{g,k,\Phi}}{d\mu_{g}}, d\mu_{g} )$ are the same. Therefore we can replace the complicated volume form $d\mu_{g,k, \Phi}$ by the product volume form  $d\mu_{g}$ and the fibre metric $g_{k}(\Phi) \otimes \sigma^k$ with $\frac{d\mu_{g,k,\Phi}}{d\mu_{g}} g_{k}(\Phi)\otimes \sigma^k$. Then we push forward the
metric $g_{k}(\Phi)\frac{d\mu_{g,k,\Phi}}{d\mu_{g}}$ to get a metric $h(k, \Phi)$ on $\textrm{Sym}^d  E$ (See
Definition \ref{2011def1}). In order to conclude the theorem, we apply the result on the
asymptotics of the Bergman kernel on $\textrm{Sym}^d E$ to the metric $h(k,\Phi)$. 

\begin{Def}\label{2011def1}

We define the hermitian metric $h(k, \Phi)$ on $\textrm{Sym}^d E$ as follows:
 \begin{equation}\label{2011eq3}h(k, \Phi)=C_{r,d}^{-1}\pi_{*} \Big(g_{k}(\Phi)\frac{d\mu_{g,k,\Phi}}{d\mu_{g}}\frac{\omega_{g}^{r-1}}{(r-1)!}\Big),\end{equation} i.e. for any $x \in X$ and $s, t \in \textrm{Sym}^d E_{x}$, we have 
 $$\langle s, t \rangle_{h(k, \Phi)}=C_{r,d}^{-1}\int_{\mathbb{P}E_{x}^*} \langle \widehat{s},\widehat{t}\, \rangle_{g_{k}(\Phi)}\frac{d\mu_{g,k,\Phi}}{d\mu_{g}} \frac{\omega_{g}^{r-1}}{(r-1)!}.$$

\end{Def}

If $h$, $\omega$ and $\Phi$ vary in a bounded family, then the metrics $h(k,\Phi)$ vary in a bounded family. More precisely, we have the following.

\begin{thm}\label{2011thm 1a}
Let $\nu_{0}$ be a fixed K\"ahler form on $X$  and $h_{0}$ be a fixed hermitian metric on $E$. For any positive numbers $l$ and $l'$ and any
positive integer $p$, there exists a positive number $C_{l,l',p}$
such that if $$||\omega||_{C^p(\nu_{0})}, ||h||_{C^{p+2}(h_{0},\nu_{0})}, ||\Phi||_{C^{p+2}(h_{0},\nu_{0})}
\leq l \,\,\,\text{and}\,\,\,$$  $$ \inf_{x \in X} ||h(x)||_{h_{0}(x)},  \inf_{x \in X}|\omega(x)^m|_{\nu_{0}(x)}\geq l',$$
then $||h(k,\Phi)||_{C^p(h_{0}, \nu_{0})} \leq C_{l,l',p}$, for $k \gg 0$.
\end{thm}
We prove Theorem \ref{2011thm 1a} in Section $5$.

\begin{proof}[Proof of Theorem  \ref{thmH1a}]


Recall that $$h(k, \Phi)=C_{r,d}^{-1}\pi_{*} \Big(g_{k}(\Phi)\frac{d\mu_{g,k,\Phi}}{d\mu_{g}}\frac{\omega_{g}^{r-1}}{(r-1)!}\Big).$$
Therefore \eqref{eqH8a}, \eqref{eqH8}, \eqref{2011eq10} and \eqref{2011eq11} imply that $$h(k, \Phi)= \textrm{Sym}^d h(I+k^{-1}(T(\Phi)+\Psi_{m-1})+O(k^{-2})).$$
 
 Let $B_{k}(h(k, \Phi),\omega)$ be the
Bergman kernel of $\textrm{Sym}^d E\otimes L^k$ with respect to the $L^2$-metric
defined by the hermitian metric $h(k,\Phi)\otimes \sigma^k$ on
$\textrm{Sym}^d E\otimes L^k$ and the volume form $\frac{ \omega^m}{m!}$ on $X$.
Therefore, if $s_{1},..., s_{N}$ is an orthonormal basis for
$H^0(X,\textrm{Sym}^d E\otimes L^k)$ with respect to the metric $L^2(h(k)\otimes
\sigma^k,\frac{ \omega^m}{m!})$, then
\begin{equation} B_{k}(h(k, \Phi),\omega)=\sum s_{i}\otimes s_{i}^{*_{h(k, \Phi)\otimes
\sigma^k}},\end{equation} We define $\widetilde{B}_{k}(h,\omega, \Phi)$
as follow
\begin{equation}\widetilde{B}_{k}(h,\omega, \Phi)=\sum s_{i}\otimes s_{i}^{*_{\textrm{Sym}^d h(I+k^{-1}\Phi)\otimes
\sigma^k}}.\end{equation} Let
$\widehat{s_{1}},....,\widehat{s_{N}} $ be the corresponding basis
for $H^{0}(\mathbb{P}E^*,\mathcal{O}_{\mathbb{P}E^*}(d)\otimes
L^{k})$. Hence,
\begin{align*}\int_{\mathbb{P}E^*} \langle \widehat{s_{i}}, \widehat{s_{j}}
\rangle_{g_{k}(\Phi) \otimes \sigma^k}d\mu_{g,\Phi, k}&=\int_{\mathbb{P}E^*}
\langle \widehat{s_{i}}, \widehat{s_{j}} \rangle_{g_{k}(\Phi) \otimes
\sigma^k}\frac{d\mu_{g,\Phi, k}}{d\mu_{g}}d\mu_{g}\\&=C_{r,d} \int_{X}\langle
s_{i}, s_{j} \rangle_{h(k,\Phi) \otimes
\sigma^k}\frac{\omega^m}{m!}\\&=C_{r,d}\delta_{ij}.\end{align*} Therefore
$\frac{1}{\sqrt{C_{r,d}}}\widehat{s_{1}},....,\frac{1}{\sqrt{C_{r,d}}}\widehat{s_{N}}
$ is an orthonormal basis for
$H^{0}(\mathbb{P}E^*,\mathcal{O}_{\mathbb{P}E^*}(d)\otimes L^{k})$
with respect to $L^2(g_{k}(\Phi) \otimes \sigma^k, d\mu_{g,\Phi,k})$. Hence, 
$$C_{r,d} \rho_{k}(g,\Phi, \omega)=C_{r,d} \rho_{k}(g_{k}(\Phi), \omega)= \sum |\hat{s_{i}}|_{g_{k}(\Phi)}^2$$ and therefore Lemma \ref{lemH4} implies that $$ \rho_{k}(g,\Phi, \omega)([v])=C_{r,d}^{-1}tr \big( \lambda_{d}(v,\textrm{Sym}^d h(I+k^{-1}\Phi))\widetilde{B}_{k}(h,\omega, \Phi)\big).$$

In order to conclude the proof, it suffices to show that there exist smooth
endomorphisms $A_{i}(h, \omega, \Phi) \in \Gamma(X,End(\textrm{Sym}^d E))$ such that
$$ \widetilde{B}_{k}(h,\omega, \Phi)\sim k^m+A_{1}(h,\omega, \Phi)k^{m-1}+\dots              .$$

The same argument as in the proof of Theorem \ref{thmH1} implies that there exist smooth
endomorphisms $B_{i} \in \Gamma(X,End(\textrm{Sym}^d E))$ such that
$$ B_{k}(h(k,\Phi),\omega)\sim k^m+B_{1}k^{m-1}+\dots              ,$$
where the first coefficient $B_{i}$ is given by \begin{equation}B_{1}(\textrm{Sym}^d h)=
\frac{i}{2\pi} \Lambda F_{(\textrm{Sym}^d E,\textrm{Sym}^d h)}+ \frac{1}{2} S(\omega) I_{\textrm{Sym}^d E}.\end{equation} 
On the other hand
\begin{align*}B_{k}(h(k,\Phi),\omega)&=\sum s_{i}\otimes s_{i}^{*_{h(k, \Phi)\otimes \sigma^k}}\\&=\sum s_{i}\otimes s_{i}^{*_{\textrm{Sym}^d h(I+k^{-1}\Phi)\otimes \sigma^k}} (I+k^{-1}\Phi)^{-1}(I+k^{-1}(T\Phi+\Psi_{m-1})+\dots)\\&=\widetilde{B}_{k}(h,\omega, \Phi)
(I+k^{-1}\Phi)^{-1}(I+k^{-1}(T\Phi+\Psi_{m-1})+\dots).\end{align*}
Therefore, \begin{align*}\widetilde{B_{k}}(h, \omega,\Phi)&=B_{k}(h(k,\Phi),\omega)(I+k^{-1}(T\Phi+\Psi_{m-1})+O(k^{-2}))^{-1}(I+k^{-1}\Phi)\\&=B_{k}(h(k,\Phi),\omega)(I+k^{-1}(\Phi-T\Phi-\Psi_{m-1})+O(k^{-2}))\\&\sim k^m+\big(B_{1}-\Psi_{m-1}+\Phi-T\Phi \big)k^{m-1}+ \dots. \end{align*}
Hence \eqref{2011eq12} imply that

\begin{align*}A_{1}(h, \omega, \Phi)&= \frac{ir}{2\pi(r+d)}  \Big (  \Lambda F_{\textrm{Sym}^d h} - \frac{1}{\textrm{rank}\big(\textrm{Sym}^d E\big)} tr( \Lambda F_{\textrm{Sym}^d h} )I  \Big )\\& +\frac{1}{2} S(\omega)I+\Phi -T\Phi= A_{1}(h,\omega)+\Phi-T\Phi.\end{align*}


Notice that Theorem \ref{2011thm 1a} implies that if $h$, $\Phi$ and $\omega$ vary in a bounded family and $\omega, h$ are bounded from below, then the metrics $h(k,\Phi)$ vary
in a bounded family. Thus the asymptotic expansion that we obtained for $\widetilde{B}_{k}(h,\omega, \Phi)$ is uniform as long as $h$, $\Phi$ and $\omega$ vary in a bounded family and $\omega$ and $h$ are bounded from below.

\end{proof}

\begin{prop}\label{propH3}
Suppose that  $\omega_{\infty} \in 2\pi c_{1}(L)$ be a K\"ahler
form with constant scalar curvature and $h_{\textrm{HE}}$ be a
Hermitian-Einstein  metric on $E$, i.e.
$\Lambda_{\omega_{\infty}} F_{(E,h_{\textrm{HE}})}=\mu I_{E}$, where
$\mu$ is the $\omega_{\infty}-$slope of the bundle $E$. We have
\begin{align*}A_{1,1}(\varphi, \eta, \Phi)&:=\frac{d}{dt}\Big|_{t=0}A_{1}(h_{\textrm{HE}}(I_{E}+t\varphi),\omega_{\infty}+it\overline{\partial}\partial
\eta, I_{\textrm{Sym}^d E}+t\Phi)\\&=\mathcal{D}^*\mathcal{D}\eta I_{\textrm{Sym}^d E} +
\frac{ir}{2\pi (r+d)}\Big( S^{d}\Lambda_{\omega_{\infty}}\overline{\partial}\partial
\varphi+\Lambda^2_{\omega_{\infty}}(F_{\textrm{Sym}^d h_{\textrm{HE}}}\wedge
i\overline{\partial}\partial \eta
)\\&\qquad-\frac{1}{R}tr\big(S^d\Lambda_{\omega_{\infty}}\overline{\partial}\partial
\varphi+\Lambda^2_{\omega_{\infty}}(F_{h_{\textrm{HE}}}\wedge
i\overline{\partial}\partial \eta )\big)\Big)+ \Phi- T \Phi,\end{align*} where
$\mathcal{D}^*\mathcal{D}$ is Lichnerowicz operator (cf.
\cite[Page 515]{D}) and $R$ is the rank of the vector bundle $\textrm{Sym}^d E$.
\end{prop}

\begin{proof}
Define
$f(t)=\Lambda_{\omega_{\infty}+it\overline{\partial}\partial
\eta}F_{(\textrm{Sym}^d h_{\textrm{HE}}(I+tS^d\varphi))}$. Therefore, we have
$$mF_{(\textrm{Sym}^d h_{\textrm{HE}}(I+tS^d\varphi))} \wedge
(\omega_{\infty}+it\overline{\partial}\partial \eta)^{m-1}= f(t)
(\omega_{\infty}+it\overline{\partial}\partial \eta)^{m}.$$
Differentiating with respect to $t$ at $t=0$, we obtain
$$m S^d\overline{\partial}\partial
\varphi \wedge \omega_{\infty}^{m-1}+m(m-1)F_{\textrm{Sym}^d h_{\textrm{HE}}}\wedge
(i\overline{\partial}\partial \eta)\wedge
\omega_{\infty}^{m-2}=f'(0)\omega_{\infty}^m+mf(0)(i\overline{\partial}\partial
\eta)\wedge \omega_{\infty}^{m-1}.$$ Since $f(0)=\mu I_{E}$, we
get
$$f'(0)=S^d\Lambda_{\omega_{\infty}}\overline{\partial}\partial \varphi
+\Lambda^2_{\omega_{\infty}}(F_{\textrm{Sym}^d h_{\textrm{HE}}}\wedge
(i\overline{\partial}\partial \eta))-\mu \Lambda_{\omega_{\infty}}(i\overline{\partial}\partial \eta)I_{E}.$$
On the other hand (cf. \cite[pp. 515, 516]{D}.)
$$\frac{d}{dt}\Big|_{t=0}S(\omega_{\infty}+it\overline{\partial}\partial
\eta)=\mathcal{D}^*\mathcal{D}\eta.$$
\end{proof}

\begin{cor}\label{corH2}
Suppose that $Aut(X,L)/\mathbb{C}^*$ is discrete and $E$ is
stable. Then the map $$A_{1,1}: \Gamma(End( E))\oplus C^{\infty}(X)\oplus     \Gamma(End(\textrm{Sym}^d E))   \rightarrow \Gamma_{0}(End(\textrm{Sym}^d E)) $$ is surjective, where $\Gamma_{0}(End(\textrm{Sym}^d E))$
is the space of smooth hermitian (with respect to $\textrm{Sym}^d h_{HE}$) endomorphisms $\Psi  $  of $\textrm{Sym}^d E$ satisfying
$\int_{X}tr(\Psi)\omega_{\infty}^m=0$.
\end{cor}

\begin{proof}

In this proof we let $F=F_{(E,h_{\textrm{HE}})}$ and $\Lambda=\Lambda_{\omega_{\infty}}$.  Define the bundle map $\widetilde{T}: \textrm{End} (\textrm{Sym}^d E) \rightarrow \textrm{End} (\textrm{Sym}^d E)$ by $\widetilde{T}\Phi=\Phi-T\Phi$. Lemma \ref{2011lem3} implies that $\ker (\widetilde{T})= S^d(End(E))$. Therefore $\ker (\widetilde{T})$ and $Im (\widetilde{T})$ are smooth subbundles of $End(\textrm{Sym}^d E)$ and as smooth bundles, we have  $$End(\textrm{Sym}^d E)=\ker (\widetilde{T}) \oplus Im (\widetilde{T}). $$

Let $\Psi \in \Gamma_{0}(End(\textrm{Sym}^d E)).$ There exist $\Psi_{1} \in \ker (\widetilde{T})$ and $\Psi_{2} \in Im (\widetilde{T})$ such that $\Psi=\Psi_{1}+\Psi_{2}$. Hence there exists $\Phi_{0} \in \Gamma(End(\textrm{Sym}^d E))$ such that \begin{equation}\label{2011eq1}\Psi_{2}=\widetilde{T} (\Phi_{0})=\Phi_{0}- T\Phi_{0} .\end{equation} 

We know that the map $\eta \in C^{\infty}_{0}\rightarrow \mathcal{D}^*\mathcal{D}\eta \in C^{\infty}_{0}$ is surjective since $Aut(X,L)/\mathbb{C}^*$ is discrete (cf. \cite[pp. 515, 516]{D}). Thus, we can find $\eta_{0} \in C^{\infty}(X) $ such that $\mathcal{D}^*\mathcal{D}\eta_{0}=\frac{1}{R}tr(\Psi).$ Note that Lemma \ref{2011lem4} and \eqref{2011eq1} imply that
$$\mathcal{D}^*\mathcal{D}\eta_{0}=\frac{1}{R}tr(\Psi)=\frac{1}{R}tr(\Psi_{1}).$$ On the other hand

 $$\frac{i}{2\pi}\big(\Lambda^2(F \wedge i\overline{\partial}\partial \eta_{0} )-\frac{1}{R}tr(\Lambda^2(F\wedge i\overline{\partial}\partial \eta_{0})\big)-\Psi_{1}+\frac{1}{R}tr(\Psi_{1}) \in \Gamma_{0}(S^dEnd(E)).$$ 
The map $$\varphi \in \Gamma_{0}(End(E))\rightarrow \frac{i}{2\pi}\Lambda\overline{\partial}\partial \varphi \in \Gamma_{0}(End(E))$$ is surjective since $E$ is simple (cf. \cite{K}). Therefore,  there exists $\varphi_{0} \in \Gamma(End(E))$ such that  
$$-\frac{i}{2\pi} S^d\Lambda\overline{\partial}\partial \varphi_{0}=\frac{i}{2\pi}\big(\Lambda^2(F \wedge i\overline{\partial}\partial \eta_{0} )-\frac{1}{R}tr(\Lambda^2(F\wedge i\overline{\partial}\partial \eta_{0})\big)-\Psi_{1}+\frac{1}{R}tr(\Psi_{1}).$$
This together with \eqref {2011eq1} imply that  $$A_{1,1}(\varphi_{0},\eta_{0}, \Phi_{0})=\Psi.$$ Note that $tr \big( \frac{i}{2\pi} S^d\Lambda\overline{\partial}\partial \varphi_{0}\big)=0$.

 \end{proof}

\section{Constructing Almost Balanced Metrics}
Let $\sigma_{\infty}$ be a hermitian metric on $L$ such that
$\omega_{\infty}=i\overline{\partial}\partial \log \sigma_{\infty} $ is a K\"ahler form with
constant scalar curvature. Let  $h_{\textrm{HE}}$ be the
corresponding Hermitian-Einstein metric on $E$, i.e.
$$\Lambda_{\omega_{\infty}} F_{(E,h_{\textrm{HE}})}=\mu I_{E},$$ where
$\mu$ is the slope of the bundle $E$. Define
\begin{equation}\label{2011eq4}\omega_{0}=i\overline{\partial}\partial \log\widehat{\textrm{Sym}^d h_{\textrm{HE}}}=di\overline{\partial}\partial \log\widehat{h_{\textrm{HE}}}.\end{equation} After tensoring by
high power of $L$, we can assume without loss of generality that
$\omega_{0}$ is a K\"ahler form on $\mathbb{P}E^*$. We fix an
integer $a \geq 4.$ In order to prove the following, we use ideas
introduced by Donaldson in (\cite[Theorem 26]{D})

\begin{thm}\label{thmH3}

Suppose $Aut(X,L)$ is discrete. There exist smooth functions
$\eta_{1},\eta_{2},...$ on $X$, smooth endomorphisms $\varphi_{1}, \varphi_{2}, \dots$  of $E$ and
$\Phi_{1},\Phi_{2},...$ smooth endomorphisms  of $\textrm{Sym}^d E$ 
such that for any positive integer
$q$ if $$\nu_{k,q}=\omega_{\infty}+i\overline{\partial}\partial
(\sum_{j=1}^q k^{-j}\eta_{j}),$$ $$h_{k,q}=h_{\textrm{HE}}(I_{E}+\sum_{j=1}^q k^{-j}\varphi_{j})$$ and
$$\Phi_{k,q}=I_{\textrm{Sym}^d E}+\sum_{j=1}^q k^{-j} \Phi_{j},$$

 then

\begin{equation}\label{eqH14}\widetilde{B}_{k}(h_{k,q},\nu_{k,q}, \Phi_{k,q})=
\frac{k^mC_{r, d}N_{k}}{V_{k}}(I_{\textrm{Sym}^d E}+\delta_{q}),\end{equation}
where $||\delta_{q}||_{C^{a+2}}=O(k^{-q-1})$. Here $V_{k}=Vol(\mathbb{P}E^*,\mathcal{O}_{\mathbb{P}E^*}(d)\otimes
L^{k})$ and $N_{k}=h^{0}(\mathbb{P}E^*,\mathcal{O}_{\mathbb{P}E^*}(d)\otimes
L^{k})$ are topological invariants.

\end{thm}

\begin{proof}
The error term in the asymptotic expansion is uniformly bounded in
$C^{a+2}$ for all $h$, $\Phi$ and $\omega$ varying in a bounded family. Therefore
there exists a positive integer $s$ depends only on $p$ and $q$
such that
\begin{align}\label{eqH15}A_{p}(h(1+\varphi),\omega+i\overline{\partial}\partial
\eta, I+\Phi)&=A_{p}(h,\omega, I)+\sum_{j=1}^{q} A_{p,j}(\varphi,\eta, \Phi)\\&
\qquad+O(||(\varphi,\eta, \Phi)||_{C^s}^{q+1}),\notag\end{align} where $
A_{p,j}$ are homogeneous polynomials of degree $j$ , depending on
$h$ and $\omega$, in $\varphi$ , $\eta$  and $\Phi$ and its covariant
derivatives. Let $\varphi_{1}, \dots, \varphi_{q}$  be smooth endomorphisms of $E$, $\Phi_{1},\dots,\Phi_{q} $ be smooth endomorphisms
of $\textrm{Sym}^d E$ and $\eta_{1},\dots,\eta_{q} $ be smooth functions on $X$. We
have \begin{align}\label{eqH16}&A_{p}(h(1+\sum_{j=1}^q
k^{-j}\varphi_{j}),\omega+i\overline{\partial}\partial (\sum_{j=1}^q
k^{-j}\eta_{j}), I+\sum_{j=1}^q
k^{-j}\Phi_{j})\\&\qquad
\qquad=A_{p}(h,\omega, I)+\sum_{j=1}^{q}b_{p,j}k^{-j}+O(k^{-q-1}),\notag\end{align}
where $b_{p,j}$'s are multi linear expression on $\varphi_{i}$, $\Phi_{i}$'s and
$\eta_{i}$'s. Hence
\begin{align}\label{eqH17}&\widetilde{B}_{k}(h(1+\sum_{j=1}^q
k^{-j}\Phi_{j}),\omega+i\overline{\partial}\partial (\sum_{j=1}^q
k^{-j}\eta_{j}), I+\sum_{j=1}^q
k^{-j}\Phi_{j})\\&=k^m+A_{1}(h,\omega, I)k^{m-1}+....\notag\\&\qquad+(A_{q}(h,\omega, I)+b_{q-1,1}+...
+b_{1,q-1})k^{m-q}+O(k^{m-q-1}).\notag\end{align} We need to
choose $\varphi_{j}$, $\Phi_{j}$ and $\eta_{j}$ such that coefficients of
$k^m,...k^{m-q}$ in the right hand side of \eqref{eqH17} are
constant. Donaldson's key observation is that $\eta_{p}$, $\varphi_{p}$ and
$\Phi_{p}$ only appear in the coefficient of $k^{m-p}$ in the form
of $A_{1,1}(\varphi_{p}, \eta_{p}, \Phi_{p})$. Hence, we can do this
inductively. Assume that we choose
$\eta_{1},\eta_{2}, \dots ,  \eta_{p-1}$ , $\varphi_{1},\varphi_{2}, \dots ,\varphi_{p-1}$ and
$\Phi_{1},\Phi_{2}, \dots ,\Phi_{p-1}$ so that  the coefficients of
$k^m,...k^{m-p+1}$ are constant. Now we need to choose $\eta_{p}$, $\varphi_{p}$
and $\Phi_{p}$ such that the coefficient of $k^{m-p}$ is constant.
This means that  we need to solve the equation
\begin{equation}\label{eqH12}A_{1,1}(\varphi_{p},\eta_{p}, \Phi_{p})-c_{p}I_{\textrm{Sym}^d E}= P_{p-1},\end{equation}
for $\varphi_{p}$, $\Phi_{p},\eta_{p}$ and the constant $c_{p}$. In this equation
$P_{p-1}$ is determined by $\varphi_{1}, \dots,\varphi_{p-1}$ $\Phi_{1},\dots,\Phi_{p-1}$ and
$\eta_{1},\dots,\eta_{p-1}$. Corollary \ref{corH2} implies that we
can always solve the equation \eqref{eqH12}.

\end{proof}

\begin{cor}\label{corH3}

For any positive integer $q$, there exist hermitian metrics
$g_{k,q}$ on $\mathcal{O}_{\mathbb{P}E^*}(d)$ and K\"ahler forms
$\nu_{k,q}$ on $X$ in the class of $2\pi c_{1}(L)$ so that
$$\rho_{k}(g_{k,q},\nu_{k,q})=\frac{k^mN_{k}}{V_{k}}(1+\epsilon_{k,q}),$$
where $||\epsilon_{k,q}||_{C^{a+2}}=O(k^{-q-1}).$ Moreover,
\begin{equation}\label{neq6}||\omega_{g_{k,q}}+k\nu_{k,q}-(\omega_{0}+k\omega_{\infty})||_{C^{a}(\omega_{0}
+k\omega_{\infty})}=O(k^{-1}),\end{equation} where $\omega_{g_{k,q}}=i\overline{\partial}\partial \log g_{k,q}$, $\omega_{0}$ is defined by \eqref{2011eq4} and $\omega_{\infty}$ is the constant scalar curvature K\"ahler metric in the class of $2\pi c_{1}(L)$.

\end{cor}

\begin{proof}
Let $H_{k,q}= \textrm{Sym}^d h_{k,q} (I+k^{-1}\Phi_{k,q})$ and $g_{k,q}=\widehat{H_{k,q}}.$ Lemma \ref {lemH4} and Theorem \ref{thmH3} imply that \begin{align*}\rho_{k}(g_{k,q},\nu_{k,q})&=\frac{N_{k}}{k^{-m}C_{r,d}V_{k}}Tr(\lambda_{d}(v^*,H_{k,q})(I_{\textrm{Sym}^d E}+\delta_{q}))\\&=\frac{N_{k}}{k^{-m}V_{k}}(1+Tr(\lambda_{d}(v^*,H_{k,q})\delta_{q}))).\end{align*}
It concludes the first part of corollary, since $H_{k,q}$ is bounded and $||\delta_{k,q}||_{C^{a+2}}=O(k^{-q-1})$. 

For the second part, define $\widetilde{\omega_{0}}=\omega_{0}+k\omega_{\infty}$ and $g_{k,q}^{\prime}=\widehat{\textrm{Sym}^d h_{k,q}}$.  Notice that $g_{k,q}=g_{k,q}^{\prime}(\Phi_{k,q})$ (c.f. \eqref{2011eq5}). Now \eqref{2011eq6} implies that 
$$F(\Phi_{k,q})= tr(\lambda_{d}\Phi_{k,q})=tr(\lambda_{d}(I+k^{-1}\Phi_{1}+\dots))=1+k^{-1}F(\Phi_{1})+O(k^{-2}).$$ Therefore \eqref{2011eq7} shows that $$\frac{g_{k,q}}{g_{k,q}^{\prime}}=1+k^{-2}F(\Phi_{1})+O(k^{-3}).$$ Hence $$\omega_{g_{k,q}}-\omega_{g_{k,q}^{\prime}}= -k^{-2}i\overline{\partial}\partial F(\Phi_{1})+O(k^{-3}),$$ which implies that 
$||\omega_{g_{k,q}}-\omega_{g_{k,q}^{\prime}}||_{C^{a}(\omega_{0})}=O(k^{-2}).$  Thus,  \begin{align*}||\omega_{g_{k,q}}+&k\nu_{k,q}-(\omega_{0}+k\omega_{\infty})||_{C^{a}(\widetilde{\omega_{0}} )} \\&\leq ||\omega_{g_{k,q}}-\omega_{g_{k,q}^{\prime}}||_{C^{a}(\widetilde{\omega_{0})}}+||\omega_{g_{k,q}}^{\prime}-\omega_{0}||_{C^{a}(\widetilde{\omega_{0}})}+k||\nu_{k,q}-\omega_{\infty}||_{C^{a}(\widetilde{\omega_{0}} )}\\&\leq ||\omega_{g_{k,q}}-\omega_{g_{k,q}^{\prime}}||_{C^{a}(\omega_{0})}+||\omega_{g_{k,q}}^{\prime}-\omega_{0}||_{C^{a}(\omega_{0})}+k||\nu_{k,q}-\omega_{\infty}||_{C^{a}(k\omega_{\infty} )}\\&=||\omega_{g_{k,q}}-\omega_{g_{k,q}^{\prime}}||_{C^{a}(\omega_{0})}+||\omega_{g_{k,q}}^{\prime}-\omega_{0}||_{C^{a}(\omega_{0})}+||\nu_{k,q}-\omega_{\infty}||_{C^{a}(\omega_{\infty} )}\\&=O(k^{-1}).\end{align*} Notice that by definition, we have \begin{align*}&||\omega_{g_{k,q}}^{\prime}-\omega_{0}||_{C^a(\omega_{0})}=O(k^{-1}),\\&||\nu_{k,q}-\omega_{\infty}||_{C^a(\omega_{\infty})}=O(k^{-1}).\end{align*}
\end{proof}

\section{Proof of Theorem \ref{2011thm 1} and \ref{2011thm 1a}}

The goal of this section is to prove theorem \ref{2011thm 1} and Theorem \ref{2011thm 1a}.
In this section, we fix a background metric $h_{0}$ on $E$ and a K\"ahler form $\nu_{0}$ on $X$. We denote the chern connections on $E$ and $\textrm{Sym}^d E$ with respect to $h_{0}$ and $\textrm{Sym}^d h_{0} $ by $\nabla$.  All norms are with respect to $h_{0}$, $\textrm{Sym}^d h_{0}$ and $\nu_{0}$. In this section we use the multi-index notation as follows:

For a multi-index $I=(i_{1}, \dots , i_{r})$, define $|I|=i_{1}+\dots +i_{r}$ and $\lambda^I=\lambda_{1}^{i_{1}} \dots \lambda_{r}^{i_{r}} $.  Let $h$ be a hermitian metric on $E$ and $e_{1}, \dots , e_{r}$ be a local holomorphic frame for $E$. We define $h_{ij}=\langle e_{i}, e_{j} \rangle _{h}$ and denote the inverse of the matrix $(h_{ij})$ by $(h^{ij})$. The hermitian metric $h$ on $E$ induces a hermitian metric $\textrm{Sym}^d h $ on $\textrm{Sym}^d E$. Any multi-index $I=(i_{1}, \dots , i_{r})$ such that $|I|=d$ defines a local holomorphic section $e^I=e_{1}^{i_{1}} \dots e_{r}^{i_{r}}$ and the set $\{ e^I | |I|=d \}$ is a holomorphic local frame for $\textrm{Sym}^d E$. Define $h_{IJ}=\langle  e^I, e^J \rangle_{\textrm{Sym}^d h}$ and as before the matrix $(h^{IJ})$ is the inverse of the matrix $(h_{IJ})$.

\begin{Def}
 A smooth function $f$ on $\mathbb{P}E^*$ is called homogenous of order $k$ with respect to the hermitian metric $h$ if there exists a local holomorphic frame $\underline{e}=(e_{1},\dots,e_{r})$ on $E$ and smooth functions $f_{IJ}$ on $X$ such that $$f(\lambda)=\frac{\sum_{|I|=|J|=k} f_{IJ} \lambda^I\overline{\lambda}^J}{(\sum h^{ij}\lambda_{i}\overline{\lambda_{j}})^k},$$ where $\lambda_{1}, \dots \lambda_{r}$ are the homogenous coordinates on the fibres with respect to the local frame $\underline{e}$. 
We define $||f||_{p, \underline{e}, h}= \max_{I,J} ||F_{IJ}||_{c^p}$

\end{Def}

From now on, let $h$ be  a hermitian metric on $E$ and $g= \widehat{\textrm{Sym}^d h}$ be the induced metric on $\mathcal{O}_{\mathbb{P}E^*}(d)$. The smooth functions $f_{1}, \dots f_{m}$ on $\mathbb{P}E^*$ and smooth hermitian endomorphisms $\Psi_{1}, \dots \Psi_{m}$ of $\textrm{Sym}^d E$ are defined in section 3 (See \eqref{eqH3}  and \eqref{eqH8}) The first step to prove  Theorem \ref{2011thm 1} is to estimate  $||f_{1}||_{p, \underline{e}, h}, \dots ||f_{m}||_{p, \underline{e}, h}.$ We establish such an estimate in Proposition \ref{2011prop 2}. The second step is to find an upper bound for  $||\Psi_{i}||_{p, \underline{e}}$ in terms of $||f_{i}||_{p, \underline{e}}$. This is the content of Theorem \ref{2011thm 2}.

The following lemma is straightforward since $X$ is compact.

\begin{lem}\label{2011lem1}
For any $x \in X $, there exists a holomorphic local frame $e_{1},\dots,e_{r}$ on $E$ around $x$ such that 
\begin{equation}\label{2011eq2}\frac{1}{2}\leq ||[\langle e_{i}, e_{j}\rangle_{h_{0}}]||_{op} \leq 2 \,\,\,\,\, \textrm{and} \,\,\,\,\,\, ||\nabla^k e_{i}|| \leq C(p)  \,\,\,\,\,\, k=0, \dots, p,\end{equation} where $C(p)$  is a constant depends only on $p, \nu_{0}$ and $h_{0}$. 
\end{lem}




\begin{lem}\label{2011lem2}(\cite[Lem.\ 5.2 ]{S})
For any positive integer $p$ there exists a constant $C$ such that for any
$(j,j)$-form $\gamma$, we have
$$||\nabla^p( \Lambda^j
\gamma)||_{\nu_{0}} \leq \frac{C}{\inf_{x \in
X}|\omega(x)^m|_{\nu_{0}(x)}}(||\gamma||_{C^p(\nu_{0})}+
||\Lambda^j
\gamma||_{C^{p-1}(\nu_{0})})(\sum_{i=1}^{m}||\omega||_{C^p(\nu_{0})}^i),$$
where $\nabla$ is the connection defined with respect to $\nu_{0}$.
\end{lem}

\begin{prop}\label{2011prop 2}
For any  $1\leq j \leq m$, the function $f_{j}$ is homogenous of order $j$ with respect to $h$ ( For definition of $f_{j}$, see \eqref{eqH3}).

 Moreover there exists a constant $C$ depends only on $p$ and $m$ such that  for any local holomorphic frame $\underline{e}=(e_{1}, \dots e_{r})$ satisfying \eqref{2011eq2}, we have    

$$||f_{i}||_{p, \underline{e}, h} \leq C \max \Big ( \Big | \Big | \frac{\nu_{0}^m}{\omega ^m} \Big |\Big|_{C^0}^p , 1\Big )  ||h||_{C^{p+2}}^j \Big (  \sum_{i=0}^m ||\omega||_{C^p}^i \Big)^{p+1}  $$
\end{prop}

\begin{proof}
Let $e_{1},...,e_{r}$ be a local
holomorphic frame for $E$ around $p$. Define $h_{ij}=\langle e_{i},
e_{j}\rangle _{h}$ and $\frac{i}{2\pi
}F_{h}=(\omega_{ij})$. Let $\lambda_{1},...,\lambda_{r}$ be the homogeneous coordinates
on the fibre. We have $$\omega_{g}=
\omega_{\textrm{FS},g}+\frac{\sum \omega_{ij} \lambda_{i}\overline{\lambda_{j}}}{\sum
h^{ij}\lambda_{i}\overline{\lambda_{j}}}.$$
Therefore, $\displaystyle \omega_{g}^{r+j-1} \wedge
\omega^{m-j}={r+j-1 \choose r-1}\omega_{\textrm{FS},g}^{r-1}
\wedge \Big(\frac{\sum \omega_{ij} \lambda_{i}\overline{\lambda_{j}}}{\sum
h^{ij}\lambda_{i}\overline{\lambda_{j}}}\Big)^j \wedge \omega^{m-j}$.
The definition of $f_{m-j}$ gives
$$f_{m-j}\omega_{g}^{r-1}\wedge
\omega^m={m \choose j}\omega_{g}^{r-1} \wedge \Big(
\Big(\frac{\sum \omega_{ij} \lambda_{i}\overline{\lambda_{j}}}{\sum
h^{ij}\lambda_{i}\overline{\lambda_{j}}}\Big)^j \wedge \omega^{m-j} \Big)$$
Hence
$$f_{m-j}\omega_{\textrm{FS},g}^{r-1}\wedge
\omega^m={m \choose j}\omega_{\textrm{FS},g}^{r-1} \wedge \Big(
\Big(\frac{\sum \omega_{ij} \lambda_{i}\overline{\lambda_{j}}}{\sum
h^{ij}\lambda_{i}\overline{\lambda_{j}}}\Big)^j \wedge \omega^{m-j} \Big).$$
Therefore,
$$\omega_{\textrm{FS},g}^{r-1}\wedge\Big(f_{m-j}
\omega^m-{m \choose j} \Big(\frac{\sum \omega_{ij} \lambda_{i}\overline{\lambda_{j}}}{\sum
h^{ij}\lambda_{i}\overline{\lambda_{j}}}\Big)^j\wedge \omega^{m-j}
\Big)=0,$$ which implies
$$f_{m-j} \omega^m={m \choose j} \Big(\frac{\sum \omega_{ij} \lambda_{i}\overline{\lambda_{j}}}{\sum
h^{ij}\lambda_{i}\overline{\lambda_{j}}}\Big)^j\wedge
\omega^{m-j}={m \choose j} \frac{\sum_{I,J} \Omega_{IJ}\lambda^{I}\overline{\lambda}^{J}}{\big(\sum
h^{ij}\lambda_{i}\overline{\lambda_{j}}\big)^j} \wedge
\omega^{m-j}.$$ 
Thus, $$f_{m-j}= \frac{\sum_{I,J} \Lambda^j \Omega_{IJ}\lambda^{I}\overline{\lambda}^{J}}{\big(\sum
h^{ij}\lambda_{i}\overline{\lambda_{j}}\big)^j}.$$
There exists a constant $C$ depends only on $p $  and $m$ such that 
$$||\Omega_{IJ}||_{C^p} \leq  C\max_{i,j} ||\omega_{ij}||_{C^p}^j \leq C ||h||_{C^{p+2}}^j.$$
Applying Lemma \ref{2011lem2}, we obtain $$||\Lambda^j \Omega_{IJ}||_{C^p}\leq C \max \Big ( \Big | \Big | \frac{\nu_{0}^m}{\omega ^m} \Big |\Big|_{C^0} , 1\Big ) \Big( ||h||_{C^{p+2}}^j+ ||\Lambda^j \Omega_{IJ}||_{C^{p-1}} \Big)\Big (  \sum_{i=0}^m ||\omega||_{C^p}^i \Big). $$ 
Note that $$\Big (\inf_{x \in X} \Big | \frac{\omega(x)^m}{\nu_{0}(x)^m} \Big|\Big)^{-1}= \sup_{x \in X} \Big | \frac{\nu_{0}(x)^m}{\omega(x) ^m} \Big|=\Big | \Big | \frac{\nu_{0}^m}{\omega ^m} \Big |\Big|_{C^0}.$$
Hence induction on $p$ concludes the proof.


\end{proof}

\begin{Def}
For any smooth function $f$ on $\mathbb{P}E^*$, there exists a unique endomorphism $\Psi(f) \in End(\textrm{Sym}^d E)$ which is hermitian with respect to $\textrm{Sym}^d h$ defined as follows: 
$$\langle s, \Psi(f) t      \rangle_{\textrm{Sym}^d h}= C_{r,d}^{-1}\int_{\mathbb{P}E_{x}^*} f \langle \hat{s}, \hat{t}      \rangle_{g} \frac{\omega_{g}^{r-1}}{(r-1)!},$$ for any $x \in M$ and $s, t \in \textrm{Sym}^d E_{x}$. Here $g= \widehat{\textrm{Sym}^d h}$ is the induced metric on $\mathcal{O}_{\mathbb{P}E^*}(d)$ by $h$.
Note that we have $\Psi_{i}= \Psi(f_{i})$.
\end{Def}

\begin{Def}
Let $\Psi$ be an endomorphism of $\textrm{Sym}^d E$ and $\underline{e}=(e_{1},\dots,e_{r})$ be a local holomorphic coordinate on $E$. Then there exist local functions $\Psi_{IJ}$ on $X$ such that 
$$\Psi e^I=\sum_{J} \Psi_{IJ} e^J.$$ We define $$||\Psi||_{p, \underline{e}}= \max ||\Psi_{IJ}||_{c^p}.$$

\end{Def}

\begin {thm}\label{2011thm 2}

There exists a constant $C$ depends only on $p$ such that for any  local holomorphic frame $\underline{e}=(e_{1},\dots,e_{r})$  satisfying \eqref{2011eq2} and  any smooth homogeneous function $f$ of  order $k$ on $\mathbb{P}E^*$, we have
 $$||\Psi(f)||_{C^p} \leq C ||f||_{p, \underline{e}}.$$

\end{thm}

The proof follows from the following lemmas.

\begin{lem}

There exists a constant $C$ depends only on $p$ such that for any  local holomorphic frame $\underline{e}=(e_{1},\dots,e_{r})$  and  any smooth homogeneous function $f$ of  order $N$ on $\mathbb{P}E^*$, we have
 $$||\Psi(f)||_{p, \underline{e}} \leq C ||f||_{p, \underline{e}}.$$

\end{lem}

\begin{proof}

We have $$\Psi(f) e^I=\sum_{J} \Psi_{IJ} e^J.$$ Let  $\lambda_{1},\dots,\lambda_{r}$ be the homogeneous coordinates
on the fibre. Suppose that $$f(\lambda)=\frac{\sum f_{IJ} \lambda^I\overline{\lambda}^J}{(\sum h^{ij}\lambda_{i}\overline{\lambda_{j}})^N}.$$Therefore 

\begin{align*}\Psi_{IJ}&= \int f(\lambda) \frac{ \lambda^I\overline{\lambda}^J}{(\sum h^{ij}\lambda_{i}\overline{\lambda_{j}})^d} \frac{\omega_{g}^{r-1}}{(r-1)!}\\&= \sum f_{KL} \int \frac{  \lambda^K\overline{\lambda}^L\lambda^I\overline{\lambda}^J}{(\sum h^{ij}\lambda_{i}\overline{\lambda_{j}})^{d+N}} \frac{\omega_{g}^{r-1}}{(r-1)!}\\&= \sum f_{KL} c_{KLIJ},\end{align*}

Here $$c_{KLIJ}=\int \frac{  \lambda^K\overline{\lambda}^L\lambda^I\overline{\lambda}^J}{(\sum h^{ij}\lambda_{i}\overline{\lambda_{j}})^{d+N}} \frac{\omega_{g}^{r-1}}{(r-1)!}.$$ An easy computation shows that $|C_{KLIJ}| \leq 1.$
Thus $$||\Psi_{IJ}||_{c^p} \leq C \max ||f_{KL}||_{c^p} = C ||f||_{p, \underline{e}}.$$

\end{proof}

\begin{lem}

There exists a constant $C$ depends only on $p$ such that for any  local holomorphic frame $\underline{e}=(e_{1},\dots,e_{r})$ satisfying \eqref{2011eq2} and smooth endomorphism $\Psi$ of $\textrm{Sym}^d E$, we have
 $$||\Psi||_{C^p} \leq C ||\Psi||_{p, \underline{e}}.$$

\end{lem}

\begin{proof}

We have $$\Psi e^I=\sum_{J} \Psi_{IJ} e^J.$$ Applying $\nabla^p$, we get

$$\sum_{i=0}^p\nabla^i \Psi \nabla^{p-i} e^I=\sum_{J} \sum_{i=0}^p\nabla^i\Psi_{IJ} \nabla^{p-i}e^J.$$
This implies that $$(\nabla^p \Psi) e^I= -\sum_{i=0}^{p-1}\nabla^i \Psi \nabla^{p-i} e^I+\sum_{J} \sum_{i=0}^p\nabla^i\Psi_{IJ} \nabla^{p-i}e^J.$$ Therefore,

\begin{align*}||(\nabla^p \Psi) e^I| &\leq \sum_{i=0}^{p-1}||\nabla^i \Psi|| ||\nabla^{p-i} e^I||+\sum_{J} \sum_{i=0}^p||\nabla^i\Psi_{IJ}|| ||\nabla^{p-i}e^J||\\ & \leq C(p)\Big(\sum_{i=0}^{p-1}||\nabla^i \Psi|| +\sum_{J} \sum_{i=0}^p||\Psi_{IJ}||_{C^p} \Big)\\& \leq CC(p) \Big(\sum_{i=0}^{p-1}||\nabla^i \Psi|| +|\Psi||_{p, \underline{e}} \Big).\end{align*}
Thus $$||\Psi||_{C^p} \leq C \Big(||\Psi||_{C^{p-1}}+ ||\Psi||_{p, \underline{e}}\Big).$$
Now we can conclude the lemma by induction on $p$.
\end{proof}

\begin{proof}[Proof of Thm.\ \ref{2011thm 1}]
applying Proposition \ref{2011prop 2} and Theorem \ref{2011thm 2}, we have
\begin{align*}
||\Psi_{i}||_{C^p} &=||\Psi(f_{i})||_{C^p} \leq C ||f_{i}||_{p, \underline{e}, h}\\& \leq C \max \Big ( \Big | \Big | \frac{\nu_{0}^m}{\omega ^m} \Big |\Big|_{C^0}^p , 1\Big )  ||h||_{C^{p+2}}^j \Big (  \sum_{i=0}^m ||\omega||_{C^p}^i \Big)^{p+1}.
\end{align*}

This concludes the proof.

\end{proof}

Let $\Phi \in \Gamma(\textrm{End} (\textrm{Sym}^d E)) $ be hermitian with respect to  $\textrm{Sym}^d h$. As in section $3$, we can define hermitian metrics $h_{t}(\Phi)=\textrm{Sym}^d h(I+t\Phi)$ and $g_{k}(\Phi)=\widehat{h_{k^{-1}}(\Phi)}$ on $\textrm{Sym}^d E$ and $\mathcal{O}_{\mathbb{P}E^*}(d)$. In the rest of this section $C$ denotes a constant depends only on $p, d, m, r, \nu_{0} $ and $h_{0}$ that might change from line to line.

\begin{lem}\label{2011lem5}

There exists a constant $C$ depends only on $p, d, m, r, \nu_{0}$ and $h_{0}$ such that $$\Big| \Big|\frac{g_{k}(\Phi)}{g}-1\Big | \Big |_{C^p} \leq C ||\Phi||_{C^p} ||h||_{C^p}^d (\inf_{x \in X} ||h(x)||_{h_{0}(x)})^{-2d}$$ for $k \gg ||\Phi||_{C^p}$.

\end{lem}

\begin{proof}

A straightforward computation Shows that $$\frac{g}{g_{k}(\Phi)}= \sum_{i=0}^{\infty} (-1)^ik^{-i} tr(\lambda_{d}\Phi^i).$$
Let $e_{1}, \dots e_{r}$ be a holomorphic local frame for $E$ satisfying \eqref{2011eq2}. Suppose that $\Phi e^I=\sum_{J}\phi_{IJ}e^{J}$, then
$$tr(\lambda_{d}\Phi^{\alpha})= \frac{\sum \lambda^{I_{0}} \overline{\lambda^{K}} h^{I_{1}K}\phi_{I_{1}I_{2}}\dots \phi_{I_{\alpha}I_{0}} }{(\sum_{ij}h^{ij}\lambda_{i}\overline{\lambda_{j}})^d}.$$ Therefore \begin{align*}||tr(\lambda_{d}\Phi^{\alpha})||_{C^p} &\leq C^{\alpha}||\Phi||^{\alpha}_{C^p}\sup_{IJ} ||h_{IJ}||_{C^p}\sup_{IJ}||h^{IJ}||_{C^0}^2\\& \leq C^{\alpha}||\Phi||^{\alpha}_{C^p} ||h||_{C^p}^d (\inf_{x \in X} ||h(x)||_{h_{0}(x)})^{-2d}.\end{align*}

Therefore 
\begin{align*}\Big| \Big|\frac{g_{k}(\Phi)}{g}-1\Big | \Big |_{C^p} &\leq \sum_{i=1}^{\infty} k^{-i}|| tr(\lambda_{d}\Phi^i)||_{C^p}\\&\leq ||h||_{C^p}^d (\inf_{x \in X} ||h(x)||_{h_{0}(x)})^{-2d}\sum_{i=1}^{\infty} C^i k^{-i}|| \Phi||_{C^p}^i\\&=||h||_{C^p}^d (\inf_{x \in X} ||h(x)||_{h_{0}(x)})^{-2d}\frac{Ck^{-1}||\Phi||_{C^p}}{1-Ck^{-1}||\Phi||_{C^p}}. \end{align*}
This concludes the proof since $$\Big| \Big|\frac{g}{g_{k}(\Phi)}\Big | \Big |_{C^p}\sim 1 \,\,\,\, \textrm{for }  \,\,\,  k \gg ||\Phi||_{C^p}.$$

\end{proof}

\begin{lem}\label{2011lem6}

 Let $H$ be a hermitian metric on $\textrm{Sym}^d E$ and $g= \widehat{H}$
 be the induced metric on $\mathcal{O}_{\mathbb{P}E^*}(d)$. Let $\chi_{1}$ and $\chi_{2}$ be smooth functions on $\mathbb{P}E^*$. Define hermitian metrics $H_{1}$ and $H_{2}$ on $\textrm{Sym}^d E$ as follows: $$\langle s, t      \rangle_{H_{i}}= \int_{\mathbb{P}E_{x}^*} \chi_{i  } \langle \hat{s}, \hat{t}      \rangle_{g} \omega_{g}^{r-1},$$ for any $x \in M$ and $s, t \in \textrm{Sym}^d E_{x}$. There exists a constant $C$ depends only on  $p, d, m, r, \nu_{0}$ and $h_{0}$ such that $$||H_{1}-H_{2}||_{C^p} \leq C ||\chi_{1}-\chi_{2}||_{C^p} ||h||_{C^p}^r  (\inf_{x \in X} ||h(x)||_{h_{0}(x)})^{-2}  .$$

\end{lem}

\begin{proof}

Let $e_{1}, \dots e_{r}$ be a holomorphic local frame for $E$ satisfying \eqref{2011eq2}. Then
$$ \langle \widehat{e^{I}}, \widehat{e^{J}}      \rangle_{g} \omega_{g}^{r-1}=d \frac{\det(h_{ij})\sum \lambda^{I} \overline{\lambda^{J}}  }{(\sum_{ij}h^{ij}\lambda_{i}\overline{\lambda_{j}})^r} d\lambda\wedge  d\overline{\lambda}.$$ Therefore

\begin{align*}    ||H_{2}-H_{1}||_{C^p} &\leq C \sum ||\int_{\textrm{Fibre}} \chi_{1} \langle \widehat{e^{I}}, \widehat{e^{J}}      \rangle_{g} \omega_{g}^{r-1}-\int_{\textrm{Fibre}} \chi_{2} \langle \hat{E_{i}}, \hat{E_{j}}      \rangle_{g} \omega_{g}^{r-1}||_{C^p} \\&\leq C ||\chi_{1}-\chi_{2}||_{C^p} ||\det(h_{ij})||_{C^p}||h^{ij}||_{C^0}\\& \leq  C ||\chi_{1}-\chi_{2}||_{C^p} ||h||_{C^p}^r  (\inf_{x \in X} ||h(x)||_{h_{0}(x)})^{-2}.\end{align*}


\end{proof}

\begin{proof}[Proof of Theorem  \ref{2011thm 1a}]

In this proof $C_{l,l^{\prime}}$ denotes a constant depends only on $p, d, m, r, l, l^{\prime} \nu_{0} $ and $h_{0}$ that might change
from line to line. Lemma \ref{2011lem5} implies that $$\Big| \Big|\frac{g_{k}(\Phi)}{g}-1\Big | \Big |_{C^{p+2}} \leq ||\Phi||_{C^{p+2}} ||h||_{C^{p+2}}^d (\inf_{x \in X} ||h(x)||_{h_{0}(x)})^{-2d} \leq Cl^{d+1}l^{\prime2}k^{-1}$$ as long as $k \gg l$. Hence 
\begin{align*} ||(\omega_{g_{k}(\Phi)}+k\omega)- (\omega_{g}+k\omega) ||_{C^p} &= \big |\big| i\bar{\partial}\partial \log \frac{g_{k}(\Phi)}{g} \big |\big|_{C^p}\leq  \big |\big|  \log \frac{g_{k}(\Phi)}{g} \big |\big|_{C^{p+2}}\\&\leq - \log  \Big (1-C   \big| \big|\frac{g_{k}(\Phi)}{g}-1\big | \big |_{C^{p+2}}   \Big)\\& \leq C  \big| \big|\frac{g_{k}(\Phi)}{g}-1\big | \big |_{C^{p+2}} \leq C_{l,l^{\prime}}k^{-1}.\end{align*}
This implies that $$\Big| \Big|\frac{d\mu_{g,k,\Phi}-d\mu_{g,k}}{d\mu_{g,k}}\Big | \Big |_{C^p} \leq C_{ l,l^{\prime}}k^{-1}.$$
Hence  \begin{align*} \Big| \Big|\frac{g_{k}(\Phi)}{g}\frac{d\mu_{g,k,\Phi}}{d\mu_{g}}-\frac{d\mu_{g,k,\Phi}}{d\mu_{g}}\Big | \Big |_{C^p}& \leq C \Big| \Big|\frac{g_{k}(\Phi)}{g}-1\Big | \Big |_{C^p}\Big| \Big|\frac{d\mu_{g,k,\Phi}}{d\mu_{g,k}}\Big | \Big |_{C^p}\\&+\Big| \Big|\frac{d\mu_{g,k,\Phi}}{d\mu_{g,k}}-\frac{d\mu_{g,k,\Phi}}{d\mu_{g,k}}\Big | \Big |_{C^p} \Big| \Big|\frac{g_{k}(\Phi)}{g}\Big | \Big |_{C^p}\Big| \Big|\frac{d\mu_{g,k}}{d\mu_{g}}\Big | \Big |_{C^p}\\& \leq C_{l,l^{\prime}}k^{-1}.\end{align*}
Note that  $$ \Big| \Big|\frac{d\mu_{g,k}}{d\mu_{g}}\Big | \Big |_{C^p} \sim 1  \,\,\, \textrm{for} \,\,\,\, k \gg 0. $$ 
Now, applying Lemma \ref{2011lem6} to the functions $\chi_{1}=\frac{d\mu_{g,k}}{d\mu_{g}}$ and $\chi_{2}=\frac{d\mu_{g,k,\Phi}}{d\mu_{g,k}}\frac{g_{k}(\Phi)}{g}$, we get $$||h(k,\Phi)-h(k)||_{C^p} \leq C_{ l,l^{\prime}} ||\chi_{1}-\chi_{2}||_{C^p} \leq C_{ l,l^{\prime}}k^{-1}.$$
This concludes the proof.

\end{proof}

\section{Proof of the main theorem}
In this section, we prove Thm.\ \ref{thm2}. In order to do that,
we want to apply \cite[Theorem 4.6.]{S}. Hence, we need to
construct a sequence of almost balanced metrics on
$\mathbb{P}E^*,\mathcal{O}_{\mathbb{P}E^*}(d)\otimes L^{\otimes
k}$. In order to apply \cite[Theorem 4.6.]{S}, we also need the following.

\begin{prop} \label{2011prop 3}(\cite[Prop.\ 7.1 ]{S})
Let $E $ be a holomorphic vector bundle over a compact K\"ahler
manifold $X$. Suppose that $X$ has no nonzero holomorphic vector
fields. If $E$ is stable, then  $\mathbb{P}E^*$ has no nontrivial
holomorphic vector fields.
\end{prop}

\begin{proof}[Proof of Thm.\ \ref{thm2}]
Since Chow stability is equivalent to the existence of balanced
metric, it suffices to show that
$(\mathbb{P}E^*,\mathcal{O}_{\mathbb{P}E^*}(d)\otimes \pi^*L^{k})$
admits balanced metric for $k \gg 0.$ Fix a positive integer $q$.
From now on we drop all indexes $q$ for simplicity. Let $\omega_{\infty}$ be a constant scalar curvature K\"ahler metric on $X$ and $\sigma_{\infty}$ be a hermitian metric on $L$ such that $i\bar{\partial}\partial \log \sigma_{\infty}=\omega_{\infty}$. Define $\sigma_{k}=\sigma_{\infty}e^{\sum_{j=1}^q k^{-j}\eta_{j}},$ where $\eta_{j}$'s are given by Thm. \ref{thmH3}. Therefore $i\bar{\partial}\partial \log \sigma_{k}=\nu_{k,q}$. For the rest of the proof, we denote $\nu_{k,q}$ by $\nu_{k}$. Let $t_{1},...,t_{N}$ be an
orthonormal basis for
$H^{0}(\mathbb{P}E^*,\mathcal{O}_{\mathbb{P}E^*}(d)\otimes L^{k})$
with respect to $L^2(g_{k}\otimes \sigma_{k}^{\otimes k},
\frac{(\omega_{g_{k}}+k\nu_{k})^{m+r-1}}{(m+r-1)!}).$ Thus,
Cor.\ \ref{corH3} implies $$\sum |t_{i}|^2_{g_{k}\otimes
\sigma_{k}^{\otimes k}}=\frac{N_{k}}{V_{k}}(1+\epsilon_{k}).$$
Define $g_{k}^{'}=\frac{V_{k}}{N_{k}}(1+\epsilon_{k})^{-1}g_{k}$.
We have $\displaystyle\sum |t_{i}|^2_{g_{k}^{'}\otimes \sigma_{k}^{\otimes
k}}=1$. This implies that the metric $g_{k}^{'}$ is the
Fubini-Study metric on $\mathcal{O}_{\mathbb{P}E^*}(d)\otimes
L^{k}$ induced by the embedding $\iota_{\underline{t}}:
\mathbb{P}E^* \rightarrow \mathbb{P}^{N-1},$ where
$\underline{t}=(t_{1},...,t_{N})$. We prove that this sequence of
embedding is almost balanced of order $q$, i.e
$$\int_{\mathbb{P}E^*} \langle t_{i},
t_{j}\rangle_{g_{k}^{'}\otimes \sigma_{k}^{\otimes
k}}\frac{(\omega_{g_{k}^{'}}+k\nu_{k})^{m+r-1}}{(m+r-1)!}=D^{(k)}\delta_{ij}+M_{ij},$$
where $M^{(k)}=[M_{ij}]$ is a trace free hermitian matrix,
$D^{(k)}=\frac{V_{k}}{N_{k}} \rightarrow C_{r,d}$ as $k \rightarrow \infty$ and
$||M^{(k)}||_{\textrm{op}}=O(k^{-q-1}).$ We have
\begin{align*}M_{ij}^{(k)}&=\int_{\mathbb{P}E^*} \langle
t_{i},t_{j}\rangle_{g_{k}^{'}\otimes\sigma_{k}^{\otimes
k}}\frac{(\omega_{g_{k}^{'}}+k\nu_{k})^{m+r-1}}{(m+r-1)!}-\frac{V_{k}}{N_{k}}\int_{\mathbb{P}E^*}\langle
t_{i}, t_{j}\rangle_{g_{k}\otimes \sigma_{k}^{\otimes
k}}\frac{(\omega_{g_{k}}+k\nu_{k})^{m+r-1}}{(m+r-1)!}\\&=\frac{V_{k}}{N_{k}}
\int_{\mathbb{P}E^*}\langle t_{i}, t_{j}\rangle_{g_{k}\otimes
\sigma_{k}^{\otimes
k}}(f_{k}(1+\epsilon_{k})^{-1}-1)\frac{(\omega_{g_{k}}+k\nu_{k})^{m+r-1}}{(m+r-1)!},\end{align*}
where
$(\omega_{g_{k}^{'}}+k\nu_{k})^{m+r-1}=f_{k}(\omega_{g_{k}}+k\nu_{k})^{m+r-1}$.
By a unitary change of basis, we may assume without loss of
generality that the matrix $M^{(k)}$ is diagonal. Thus
$$||M^{(k)}||_{\textrm{op}} \leq
\frac{V_{k}}{N_{k}}||f_{k}(1+\epsilon_{k})^{-1}-1||_{L^{\infty}}.$$
On the other hand,
\begin{align*}||\omega_{g_{k}^{'}}-\omega_{g_{k}}||_{C^0(\omega_{0})}&=||\overline{\partial}\partial
\log(1+\epsilon_{k})||_{C^0(\omega_{0})}\leq
||\log(1+\epsilon_{k})||_{C^2(\omega_{0})}\\
&\leq
-\log(1-C||\epsilon_{k}||_{C^2(\omega_{0})})=
O(k^{-q-1}).
\end{align*}
 Therefore,
\begin{align*}||f_{k}-1||_{\infty}=\Big|\frac{\omega^{m+r-1}_{g_{k}^{'}}-\omega_{g_{k}}^{m+r-1}}{\omega_{g_{k}}^{m+r-1}}\Big|&=\Big|\frac{\omega^{m+r-1}_{g_{k}^{'}}-\omega_{g_{k}}^{m+r-1}}{\omega_{0}^{m+r-1}}\frac{\omega_{0}^{m+r-1}}{\omega_{g_{k}}^{m+r-1}}\Big|\\ &\leq Ck^{-q-1}\Big|\frac{\omega_{0}^{m+r-1}}{\omega_{g_{k}}^{m+r-1}}\Big|.
\end{align*}
This implies  that $||f_{k}-1||_{\infty}\leq Ck^{-q-1}$, since
$\displaystyle\Big|\frac{\omega_{0}^{m+r-1}}{\omega_{g_{k}}^{m+r-1}}\Big|$ is
bounded. Hence
$||f_{k}(1+\epsilon_{k})^{-1}-1||\leq C^{'}k^{-q-1}$. Therefore
$||M^{(k)}||_{\textrm{op}} =O(k^{-q-1})$.
Prop. \ref{2011prop 3} implies that $\mathbb{P}E^*$ has no
nontrivial holomorphic vector fields. On the other hand $$||\log (g_{k}\otimes \sigma_{k}^{\otimes k})-\log (\widehat{h_{HE}}\otimes \sigma_{\infty}^{\otimes k})||_{C^{a+2}(\widetilde{\omega}_{0})}=O(1).$$Therefore, applying \cite[Theorem 4.6.]{S} and \eqref{neq6} conclude the proof. 
\end{proof}



\begin{thebibliography}{HHHH}

\bibitem[BBS] {BBS} R. Berman, B. Berndtsson, J. Sj\"ostrand, A direct approach to Bergman kernel asymptotics for positive line bundles,     arXiv:math/0506367v2


\bibitem[C] {C}D. Catlin, The Bergman kernel and a theorem of Tian, in {\it Analysis and geometry in several complex variables (Katata,1997)}, 1--23, Birkh\"auser, Boston, Boston, MA.



\bibitem[D] {D}S. K. Donaldson, Scalar curvature and projective embeddings. I, J. Differential Geom. {\bf 59} (2001), no.~3, 479--522.





\bibitem[K] {K} S. Kobayashi, Differential geometry of complex vector bundles. Publications of the Mathematical Society of Japan, 15.Kanô Memorial Lectures, 5. Princeton University Press, Princeton,NJ; Iwanami Shoten, Tokyo, 1987. xii+305 pp. ISBN: 0-691-08467-X
\bibitem[Lu]{Lu} Z. Lu, On the lower order terms of the asymptotic expansion of Tian-Yau-Zelditch, Amer. J. Math. {\bf 122} (2000), no.~2, 235--273. MR1749048 (2002d:32034)

\bibitem[L] {L}H. Luo, Geometric criterion for Gieseker-Mumford stability of polarized manifolds, J. Differential Geom. {\bf 49} (1998), no.~3, 577--599.

\bibitem[M] {M}I. Morrison, Projective stability of ruled surfaces, Invent.Math.{\bf 56} (1980), no.~3, 269--304.

\bibitem[PS1] {PS1} D. H. Phong\ and\ J. Sturm, Stability, energy functionals, and K\"ahler-Einstein metrics, Comm. Anal. Geom. {\bf 11} (2003), no.~3, 565--597.



\bibitem[S] {S}  R. Seyyedali,  Balanced metrics and Chow stability of projective bundles over K\"ahler manifolds, Duke Math. J. Volume 153, Number 3 (2010), 573-605. 
   


\bibitem[UY] {UY}K. Uhlenbeck\ and\ S.-T. Yau, On the existence of Hermitian-Yang- Mills connections in stable vector bundles, Comm. Pure Appl. Math. {\bf 39} (1986), no.~S, suppl., {\rm S}257--{\rm S}293.


\bibitem[W] {W}X. Wang,  Canonical metrics on stable vector bundles. Comm.Anal. Geom.  13  (2005), no. 2, 253--285.

\bibitem[Z] {Z}S. Zelditch, Szeg\H o kernels and a theorem of Tian, Internat.Math. Res. Notices {\bf 1998}, no.~6, 317--331.

\bibitem[Zh] {Zh}S. Zhang, Heights and reductions of semi-stable varieties, Compositio Math. {\bf 104} (1996), no.~1, 77--105.

\end{thebibliography}
\end{document}